\documentclass[reqno]{amsart}

\usepackage{amsmath}
\usepackage{amssymb}
\usepackage{amsthm}
\usepackage{graphicx}

\newtheorem{theorem}{Theorem}[section]

\newtheorem{corollary}[theorem]{Corollary}
\newtheorem{lemma}[theorem]{Lemma}
\theoremstyle{definition}

\newtheorem{example}[theorem]{Example}

\newtheorem{remark}[theorem]{Remark}

\newcommand{\rpm}{{\rho_{+\leftrightarrow -}}}

\makeatletter
\@namedef{subjclassname@2020}{%
  \textup{2020} Mathematics Subject Classification}
\makeatother

\begin{document}


\title[]{Quandles versus symmetric quandles for oriented links}
\author[]{Kanako Oshiro}
\address{Department of Information and Communication Sciences, 
Sophia University, 7-1 Kioicho, Chiyoda-ku, Tokyo 102-8554, Japan}
\email{oshirok@sophia.ac.jp}

\maketitle

\begin{abstract}
Given a quandle, we can construct a symmetric quandle called the symmetric double of the quandle. 
We show that the (co)homology groups of a given quandle are isomorphic to those of its symmetric double.
Moreover, quandle coloring numbers and quandle cocycle invariants of oriented links and oriented surface-links can be interpreted by using symmetric quandles. 
\end{abstract}

\section{Introduction}
A quandle is an algebraic system independently introduced by D. Joyce \cite{Joyce82} and S. Matveev \cite{Matveev82} in 1982. There are several studies using quandles in knot theory. 
Especially quandle cocycle invariants \cite{CarterJelsovskyKamadaLangfordSaito03} (see also \cite{CarterKamadaSaitobook}), introduced by J. S. Carter et al. in 2003, are very useful for studies of oriented links and oriented surface-links, refer to \cite{ Iwakiri10, KawamuraOshiroTanaka, Nosaka, OshiroSatoh, OshiroTanaka, Satoh07, Satoh16, SatohShima04, Tanaka07} for example. 
Here, we note that in order to define quandle cocycle invariants, it is essential that links or surface-links are oriented.  
In 2007, S. Kamada \cite{Kamada2006} (see also \cite{KamadaOshiro10}) introduced symmetric quandle cocycle invariants for links and surface-links.
In order to define symmetric quandle cocycle invariants, it is not required that links or surface-links are oriented or orientable, and therefore, this enabled us to consider quandle cocycle invariants for links or surface-links which are not necessarily oriented or orientable. 
There are several studies using symmetric quandle colorings and symmetric quandle cocycle invariants, see \cite{CarterOshiroSaito10, JangOshiro,  Kamada2014, KamadaOshiro10, Oshiro10, Oshiro11}  for example.

In this paper, we focus on a symmetric quandle which is called the symmetric double of a given quandle. 
We show that the (co)homology groups of a given quandle are isomorphic to those of the symmetric double of the quandle (Section~\ref{sec:4}).
We also show that for a given quandle, the quandle coloring numbers and the quandle cocycle invariants of oriented (surface-)links can be interpreted by using the symmetric double of the quandle (Sections~\ref{sec:5} and \ref{sec:6}).

This paper is organized as follows.
In Section~\ref{sec:2}, we  review the definitions of quandles, quandle (co)homology groups, and link invariants using quandles. 
In Section~\ref{sec:3}, we  review the definitions of symmetric quandles, symmetric quandle (co)homology groups, and link invariants using symmetric quandles.
In Section~\ref{sec:4}, we define the symmetric double of a given quandle, and we  prove that the (co)homology groups of a quandle are isomorphic to those of its symmetric double.  
In Section~\ref{sec:5}, we discuss oriented link invariants using quandles. Especially, we show that 
quandle coloring numbers, quandle homology invariants and quandle cocycle invariants for oriented links can be interpreted by using symmetric quandles.
In Section~\ref{sec:6}, we summarize the same properties as shown in Section~\ref{sec:5} for oriented surface-links.

\section{Quandles and invariants for oriented links}\label{sec:2}
In this section, we  review the definitions of quandles, quandle (co)homology groups, link invariants using quandles. 
\subsection{Quandles}
A \textit{quandle} \cite{Joyce82, Matveev82}, see also \cite{FennRourke92},  is a set $X$ with a binary operation $*: X \times X \to X$ satisfying the following axioms.
\begin{itemize}
\item[(Q1)] For any $x \in X$, $x * x =x$.
\item[(Q2)] For any $x, y \in X$, there exists a unique element $z\in X$ such that $z * y=x$.
\item[(Q3)] For any $x,y,z \in X$, $(x * y) * z = (x*  z) * (y * z)$.
\end{itemize}
The unique element $z$ in the second axiom is denoted by $x *y^{-1}$. 
For a quandle $X$, when we specify the quandle operation $*$, we also denote  by $(X, *)$ the quandle $X$. 

The proof of the next lemma is straightforward, and we leave it to the reader.
\begin{lemma}\label{lem:1}Let $(X,*)$ be a quandle.
For any $x,y,z\in X$ and $\varepsilon, \delta \in \{\pm 1\}$, 
it holds that 
\[
(x*y^{\varepsilon })*z^{\delta} = (x*z^{\delta})*(y*z^\delta)^\varepsilon,
\]
where in this paper $x*y^{+1}$ means $x*y$.
\end{lemma}

\begin{example}\label{ex:quandle1}
A {\it trivial quandle}  is a non-empty set $S$ with the trivial operation, that is, $x*y=x$ for any $x,y\in S$.
A trivial quandle consisting of $n$ elements is denoted by $T_n$.
\end{example}
\begin{example}\label{ex:quandle2}
The {\it dihedral quandle} of order $n$ for an integer $n\geq 3$  is $\mathbb Z/ n\mathbb Z$ with the operation $x*y=2y-x$ for any $x,y \in \mathbb Z/ n\mathbb Z$. We denote it by $R_n$. 
\end{example}
\begin{example}\label{ex:quandle3}
A {\it conjugation quandle} is a group $G$ with the operation $g*h= h^{-1} g h$ for any $g,h \in G$. 
We denote it by ${\rm Conj}\,G$.
\end{example}

\subsection{Quandle (co)homology theories}\label{subsec:quandlehomology}
Carter et al. \cite{CarterJelsovskyKamadaLangfordSaito03} defined homology groups of quandles, which are deeply related to homology groups of racks due to Fenn et al. \cite{FennRourkeSanderson95}, and various versions are introduced so far, see \cite{CarterElhamdadiGranaSaito05, CarterKamadaSaitobook}. 
In this paper, we adopt the definition of the homology group of a quandle $X$ with an $X$-set $Y$. 
We note that when $Y$ is a singleton, the homology group concides with that defined in \cite{CarterJelsovskyKamadaLangfordSaito03}, and when $Y=X$, the homology group coincides with that defined in \cite{CarterKamadaSaitobook}.

The {\it associated group} of a quandle $(X,*)$ is 
\[
G_X=\langle x\in X ~|~x*y= y^{-1} x y ~(x,y\in X) \rangle,
\]
where we note that any element of $X$ can be regarded as an element of $G_X$ through the map $X \to G_X; x \mapsto x$.
An {\it $X$-set} is a set $Y$ equipped with a right action by the associated group $G_X$. 
In this paper, we use the same symbol for the operation of a given quandle $X$ and the action on a given  $X$-set $Y$ by $G_X$, that is, we denote by $r*g$ the image of an element $r\in Y$ acted on by $g \in G_X$.

Let $(X,*)$ be a quandle and $Y$ an $X$-set.
The same proof as in Lemma~\ref{lem:1} works for the next lemma, and we leave it to the reader.
\begin{lemma}\label{lem:2}
For any $r\in Y$, $x,y\in X$ and $\varepsilon, \delta \in \{\pm 1\}$, 
it holds that 
\[
(r*x^{\varepsilon })*y^{\delta} = (r*y^{\delta})*(x*y^\delta)^\varepsilon.
\]
\end{lemma}

Let $C_n(X)_Y$ be the free abelian group generated by $(n+1)$-tuples $(r,x_1, \ldots , x_n)\in Y \times X^n$ when $n$ is a positive integer, and let $C_n(X)_Y=\{0\}$ otherwise. The boundary homomorphism $\partial_n : C_n (X)_Y \to C_{n-1} (X)_Y$ is defined by 
\[
\begin{array}{rl}
\partial_n (r,x_1, \ldots , x_n) =& \displaystyle \sum_{i=1}^n (-1)^{i}\big\{ (r, x_1,  \ldots , \hat{x_i}, \ldots, x_n)\\
&\ \ \ \ - (r*x_i, x_1*x_i,\ldots , x_{i-1}*x_i, \hat{x_i}, x_{i+1}, \ldots , x_n)\big\}
\end{array}
\]
for $n>1$ and $\partial_n=0$ otherwise, where $\hat{x_i}$ represents that $x_i$ is removed. Then $C_* (X)_Y = \{ C_n (X)_Y , \partial_n\}_{n\in \mathbb Z}$ is a chain complex.

Let $D_n(X)_Y$ be the subgroup of $C_n(X)_Y$ generated by the elements of 
\[
\{ (r, x_1, \ldots , x_n) \in Y\times X^n ~|~ x_i = x_{i+1} \mbox{ for some $i\in \{1,\ldots , n-1\}$}\}.
\]
Then $D_*(X)_Y = \{ D_n(X)_Y, \partial _n\}_{n\in \mathbb Z}$ is a subchain complex of $C_* (X)_Y$.
Let $C_n^{\rm Q} (X)_Y= C_n (X)_Y/ D_n (X)_Y$, and we denote by $\partial_n^{\rm Q}$ the induced boundary homomorphism $\partial_n^{\rm Q}: C_n^{\rm Q} (X)_Y \to C_{n-1}^{\rm Q} (X)_Y$. 
The quotient chain complex $C_*^{\rm Q} (X)_Y=\{ C_n^{\rm Q} (X)_Y, \partial_n^{\rm Q} \}_{n\in \mathbb Z}$ leads to the {\it quandle homology group} of $X$ and $Y$ defined by $H_n^{\rm Q} (X)_Y= {\rm Ker} \partial_n^{\rm Q}/ {\rm Im} \partial_{n+1}^{\rm Q}$.

For an abelian group $A$, we define the chain complex $C_*^{\rm Q} (X;A)_Y =\{ C_n^{\rm Q} (X)_Y\otimes A, \partial_n^{\rm Q} \otimes {\rm id} \}_{n\in \mathbb Z}$. 
The homology groups are denoted by $H_*^{\rm Q} (X;A)_Y$.
Note that when $A=\mathbb Z$, $H_*^{\rm Q} (X; \mathbb Z)_Y$ represents $H_*^{\rm Q} (X)_Y$.
We also define the cochain group $C_{\rm Q}^n (X; A)_Y$ and the coboundary homomorphism $\delta_{\rm Q}^n : C_{\rm Q}^n (X; A)_Y \to C_{\rm Q}^{n+1} (X; A)_Y$  
by $C_{\rm Q}^n (X; A)_Y = {\rm Hom}(C_n^{\rm Q}(X)_Y, A)$ and $\delta_{\rm Q}^{n} (\theta)=  \theta \circ \partial ^{\rm Q}_{n+1}$, respectively. 
The {\it quandle cohomology group with coefficients in $A$} is defined  
by $H_{\rm Q}^n (X; A)_Y={\rm Ker} \delta_{\rm Q}^n/{\rm Im} \delta_{\rm Q}^{n-1}$.
We denote by $Z_{\rm Q}^n (X; A)_Y$ and $B_{\rm Q}^n (X; A)_Y$ the cocycle group ${\rm Ker} \delta_{\rm Q}^n$ and the coboundary group ${\rm Im} \delta_{\rm Q}^{n-1}$, respectively.

\subsection{Quandle colorings} 
A {\it link} is a disjoint union of circles embedded in $\mathbb R^3$.
Two links are said to be {\it equivalent} if there exists an orientation-preserving self-homeomorphism of $\mathbb R^3$ which maps one link onto the other.
A {\it diagram} of a link is its image, via a generic projection from $\mathbb R^3$ to $\mathbb R^2$, equipped with the height information around the crossings. 
The height information of a diagram of a link is given by removing regular neighborhoods of the lower crossing points, and thus, the diagram is regarded as a disjoint union of connected compact parts. We call each connected component an {\it arc} of the diagram. 
The $2$-dimensional space $\mathbb R^2$ is separated into several connected regions by a diagram. We call each connected region a {\it complementary region} of the diagram.

Let $(X,*)$ be a quandle.
Let $(D, o)$ be an oriented link diagram, that is, a pair of  a link diagram $D$ and an orientation $o$ of $D$. 
In this paper, we often use normal orientations of arcs for representing the orientation $o$ of $D$ as in Figure~\ref{normalori}, where for an arc $\alpha$ of $D$, the pair of the orientation $o$  and the normal orientation $n_\alpha$ of the arc always forms the right-handed orientation of $\mathbb R^2$.  
\begin{figure}[t]
 \begin{center}
\includegraphics[width=15mm]{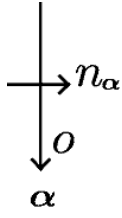}
\end{center}
\caption{The orientation and the normal orientation}
\label{normalori}
\end{figure}

\begin{figure}[t]
 \begin{center}
\includegraphics[width=35mm]{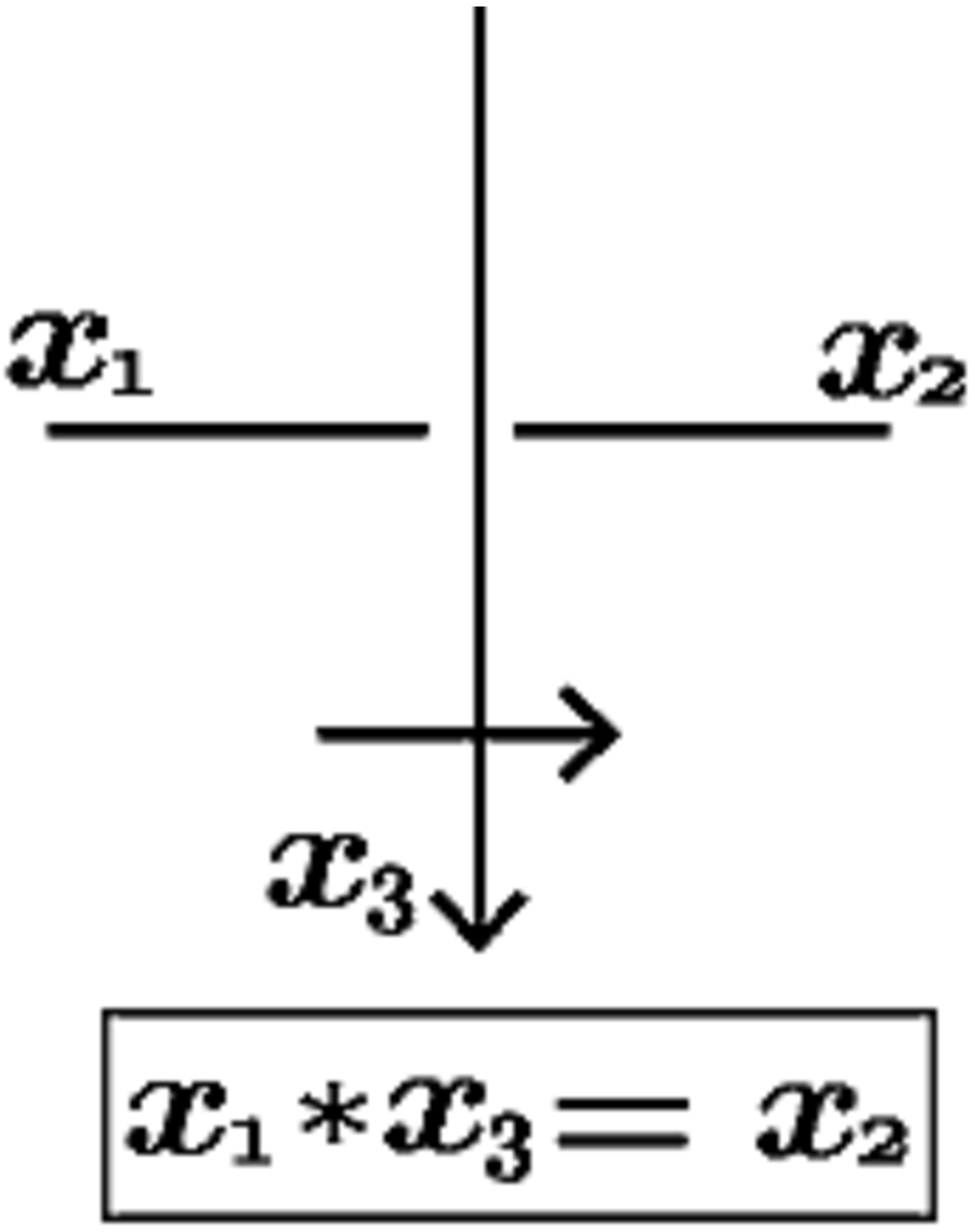}\ \hspace{2cm}
\includegraphics[width=25mm]{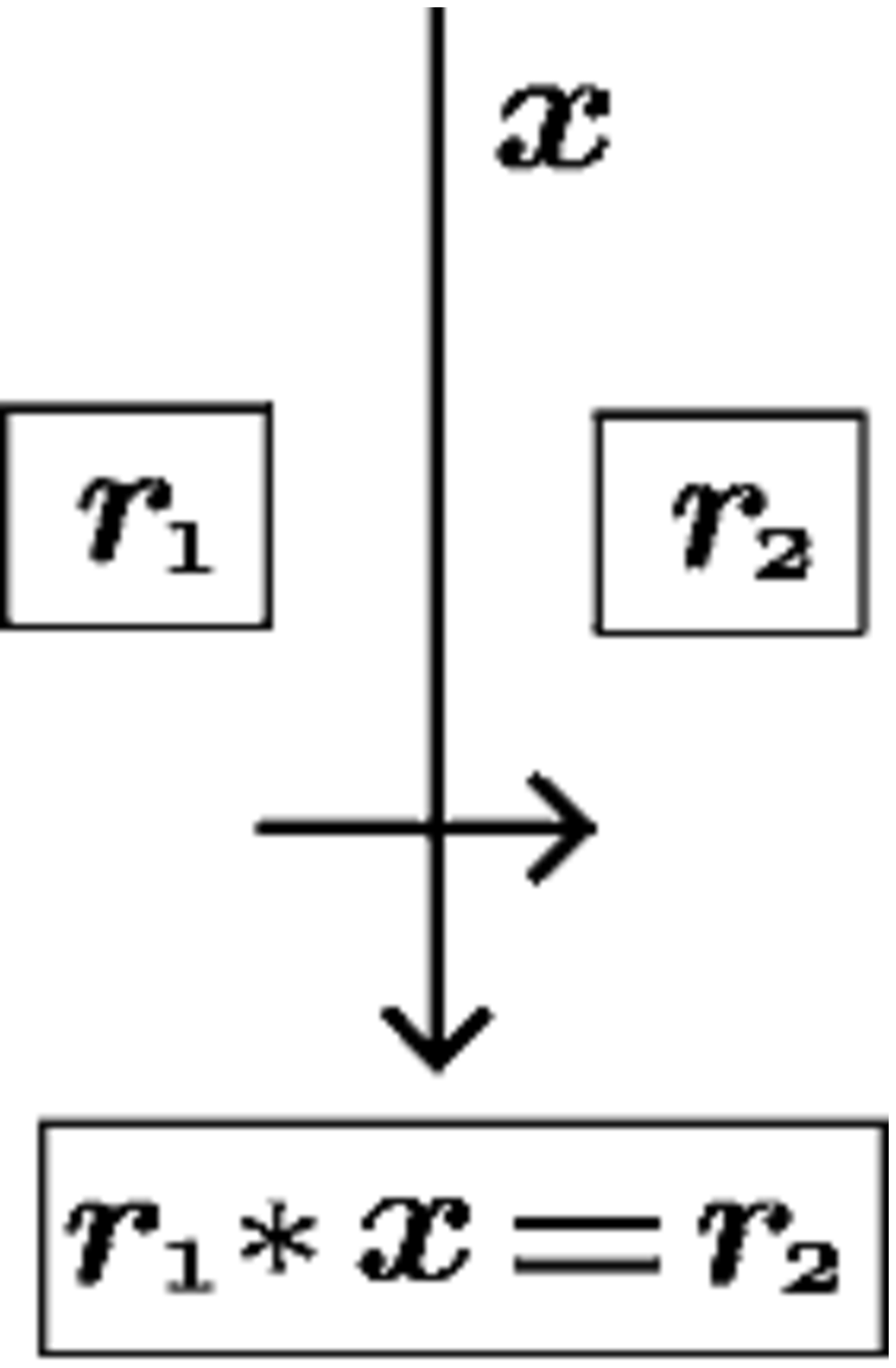}
\end{center}
\caption{The conditions of quandle colorings}
\label{quandlecoloring1}
\end{figure}

An {\it $X$-coloring} of $(D, o)$ is an assignment of  an element of $X$ to each arc  satisfying the following condition. 
\begin{itemize}
\item Suppose that two arcs $\alpha_1$ and $\alpha_2$ which are under-arcs at a crossing $\chi$ are labeled by $x_1$ and $x_2$, respectively, and the over arc, say $\alpha_3$,  of $\chi$ is labeled by $x_3$. 
We assume that the normal orientation of $\alpha_3$ points from $\alpha_1$ to $\alpha_2$.
Then $x_1 * x_3 =x_2$ holds. See the left of Figure~\ref{quandlecoloring1}.
\end{itemize}
We denote by ${\rm Col}_{X}(D, o)$ the set of $X$-colorings of $(D, o)$. 
We then have the following lemma.
\begin{lemma}[cf. \cite{CarterJelsovskyKamadaLangfordSaito03, CarterKamadaSaitobook}]
Let $(D, o)$ and $(D', o')$ be diagrams of oriented links.
If $(D, o)$ and $(D', o')$ represent the same oriented link, then there exists a one-to-one correspondence between ${\rm Col}_{X}(D,o)$ and ${\rm Col}_{X}(D',o')$. 
\end{lemma}
 This implies that $\# {\rm Col}_{X}(D, o)$, that is the cardinality of ${\rm Col}_{X}(D, o)$,  is an invariant for oriented links, and hence, we also denote it by $\# {\rm Col}_{X}(L,o)$ and call it the {\it quandle coloring number} of $(L,o)$ when $X$ is finite, where $(L,o)$ is an oriented link which $(D,o)$ represents.
We note that in this paper, when a diagram $(D,o)$ represents an oriented link, we use the same symbol $o$ for the orientation of the oriented link.

 Let $Y$ be  an $X$-set.
An {\it $X_Y$-coloring} of $(D,o)$ is an $X$-coloring of $(D,o)$ with an assignment of an element of $Y$ to each complementary region of $D$ satisfying  the following condition. 
\begin{itemize}
\item Suppose that two complementary regions $\gamma_1$ and $\gamma_2$ are adjacent over an arc $\alpha$, and  $\gamma_1$, $\gamma_2$ and $\alpha$ are labeled by $r_1$, $r_2$ and $x$, respectively. We assume that the normal orientation of $\alpha$ points from $\gamma_1$ to $\gamma_2$. Then $r_1*x=r_2$ holds. See the right of Figure~\ref{quandlecoloring1}. 
\end{itemize}
We denote by ${\rm Col}_{X_Y}(D,o)$ the set of $X_Y$-colorings of $(D,o)$. Then we have the following lemma.
\begin{lemma}[cf.  \cite{CarterJelsovskyKamadaLangfordSaito03, CarterKamadaSaitobook}]
Let $(D,o)$ and $(D',o')$ be diagrams of oriented links.
If $(D,o)$ and $(D',o')$ represent the same oriented link, then there exists a one-to-one correspondence between ${\rm Col}_{X_Y}(D,o)$ and ${\rm Col}_{X_Y}(D',o')$. 
\end{lemma}
 This implies that  $\# {\rm Col}_{X_Y}(D,o)$,  that is the cardinality of ${\rm Col}_{X_Y}(D,o)$, is an invariant for oriented links.

\subsection{Quandle cocycle invariants of oriented links}\label{Quandle cocycle invariants of oriented links}
In this subsection, we review the definition of a quandle cocycle invariant of an oriented link. 
We remark that several variations for quandle cocycle invariants are introduced so far, see \cite{CarterJelsovskyKamadaLangfordSaito03, CarterKamadaSaitobook}, and we adopt the definition of a  shadow quandle cocycle invariant with a quandle $X$ and an $X$-set $Y$.
We note that when $Y$ is a singleton, the shadow quandle cocycle invariant coincides with the quandle cocycle invariant with $X$ defined in \cite{CarterJelsovskyKamadaLangfordSaito03}, and when $Y=X$, the shadow quandle cocycle invariant coincides with that defined in \cite{CarterKamadaSaitobook}.

Let $(X,*)$ be a quandle and $Y$ an $X$-set.
Let $(D,o)$ be a diagram of an oriented link $(L,o)$ and $C$ an $X_Y$-coloring of $(D,o)$. 
For each crossing $\chi$ of $D$, we extract a weight $w_\chi$ as follows.
Let $\alpha_1$ and $\alpha_2$ denote the under-arc and the over-arc of $\chi$, respectively,  such that the normal orientation of $\alpha_2$ points from $\alpha_1$, and $\gamma$ the complementary region of $D$ around $\chi$ which the normal orientations of $\alpha_1$ and $\alpha_2$ point from.
We suppose that $\gamma$, $\alpha_1$ and $\alpha_2$ are labeled by $r$, $x_1$ and $x_2$, respectively. 
Then the {\it weight} of the crossing $\chi$ is  $w_{\chi}=\varepsilon (r, x_1, x_2)$, where $\varepsilon=+$ if $\chi$ is  positive and $\varepsilon=-$ if $\chi$ is  negative, see Figure~\ref{weight-quandle}.
We note that a crossing $\chi$ is {\it positive} if the pair of the normal orientations of the over-arc and an under-arc of $\chi$ matches the right-handed orientation of $\mathbb R^2$, and $\chi$ is {\it negative} otherwise.
Let $W^{X_Y}((D,o), C)$ denote the sum of the weights of all the crossings, and then, $W^{X_Y}((D,o),C)$ is a $2$-cycle of $C_*^{\rm Q}(X)_Y$.
Define
\[
\mathcal{W}^{X_Y} (D,o) = \{ [W^{X_Y}((D,o),C)] \in H_2^{\rm Q}(X)_Y ~|~ C\in {\rm Col}_{X_Y} (D,o)\}
\]
as a multi-set, where a multi-set is a collection of elements which may appear more than once.
\begin{figure}[t]
 \begin{center}
\includegraphics[width=80mm]{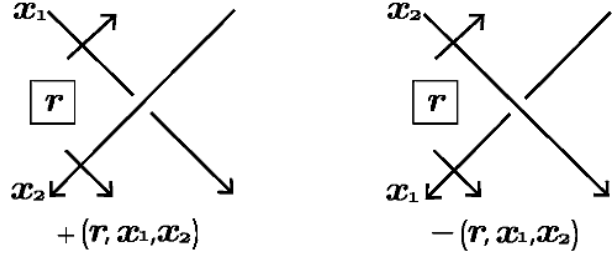}
\end{center}
\caption{The weight of a crossing}
\label{weight-quandle}
\end{figure}

For an abelian group $A$, let $\theta: C_2^{\rm Q}(X)_Y \to A$ be a quandle $2$-cocycle. We define the multi-set $\Phi_\theta^{X_Y}(D)$ by 
\[
\Phi_\theta^{X_Y}(D,o) = \{ \theta(W^{X_Y}((D,o),C)) \in A ~|~ C\in {\rm Col}_{X_Y} (D,o)\}.
\] 

\begin{theorem}[cf. \cite{CarterJelsovskyKamadaLangfordSaito03, CarterKamadaSaitobook}]
The multi-sets $\mathcal{W}^{X_Y}(D,o)$ and $\Phi_\theta^{X_Y} (D,o)$ are invariants of the link type of $(D,o)$.
\end{theorem}
 We also denote these invariants by $\mathcal{W}^{X_Y} (L,o)$ and $\Phi_\theta^{X_Y}(L,o)$, respectively, and call them the {\it quandle homology  invariant} of $(L,o)$ by $X$ and $Y$ and the {\it quandle cocycle invariant} of $(L,o)$ by $X, Y$ and $\theta$, respectively.

\section{Symmetric quandles and invariants for unoriented links}\label{sec:3}
In this section, we  review the definitions of symmetric quandles, symmetric quandle (co)homology groups, link invariants using symmetric quandles.
We note again that while link invariants using quandles need orientations of links, link invariants using symmetric quandles do not use orientations of links (or surface-links). This enabled us to consider quandle invariants for links  (or surface-links) which are not necessarily oriented or orientable.

\subsection{Symmetric quandles}
A \textit{symmetric quandle} \cite{Kamada2006}, see also \cite{KamadaOshiro10}, is a pair $(X, \rho)$ of a quandle $(X,*)$ and a good involution $\rho$ of $X$, where an involution $\rho :X \to X$ is said to be {\it good } if 
\[\mbox{
(GI1) $\rho(x*y) = \rho(x)*y$~~ and  ~~
(GI2) $x*\rho(y)=x*y^{-1}$}
\]
for any $x,y\in X$.
For a symmetric quandle $(X,\rho)$, when we specify the symmetric quandle operation $*$, we also denote  by $(X,  \rho, *)$ the symmetric quandle $(X,\rho)$.

\begin{example}\label{ex:symmetricquandle1}
Let $X$ be a trivial quandle.
Any involution $\rho: X\to X$ is a good involution of $X$.
\end{example}
\begin{example}\label{ex:symmetricquandle3} 
Let $G$ be a group.
The inversion $\rho : {\rm Conj}\,G \to {\rm Conj}\,G; g\mapsto g^{-1}$ is a good involution of the conjugation quandle ${\rm Conj}\,G$.
\end{example}

\subsection{Symmetric quandle (co)homology theories} 
Kamada  \cite{Kamada2006}, see also \cite{KamadaOshiro10} in detail, defined (co)homology groups of symmetric quandles. In this subsection, we review the definition of a symmetric quandle homology group.

The {\it associated group} of  a symmetric quandle $(X,\rho, *)$ is 
\[
G_{(X,\rho)}=\langle x\in X ~|~x*y= y^{-1} x y ~(x,y\in X), ~~\rho(x)=x^{-1} ~(x \in X) \rangle.
\]
An {\it $(X,\rho)$-set} is a set $Y$ equipped with a right action of the associated group $G_{(X,\rho)}$. 
In this paper, we use the same symbol for the operation of a given symmetric quandle $(X,\rho)$ and the action on a given  $(X,\rho)$-set $Y$ by $G_{(X,\rho)}$, that is, we denote by $r*g$ the image of an element $r\in Y$ acted on by $g \in G_{(X,\rho)}$.

Let $(X,\rho,*)$ be a symmetric quandle and $Y$ be an $(X,\rho)$-set.
The chain group $C_n(X)_Y$, the boundary homomorphism $\partial_n : C_n (X)_Y \to C_{n-1} (X)_Y$, and the subgroup $D_n(X)_Y$ of $C_n(X)_Y$ are defined as in Subsection~\ref{subsec:quandlehomology}.

Let $D_n(X,\rho)_Y$ be the subgroup of $C_n(X)_Y$ generated by the elements of 
\[
\begin{array}{l}
\displaystyle \bigcup_{i=1}^n\Big\{ (r, x_1, \ldots , x_n)+(r*x_i, x_1*x_i, \ldots , x_{i-1}*x_i, \rho(x_i), x_{i+1},\ldots , x_n) \\
\hspace{8cm}\Big|~ r \in Y, x_1, \ldots , x_n\in X \Big\} .
\end{array}
\]
Then $D_*(X,\rho)_Y = \{ D_n(X,\rho)_Y, \partial _n\}_{n\in \mathbb Z}$ is a subchain complex of $C_* (X)_Y$.
Let $C_n^{\rm SQ} (X,\rho)_Y= C_n (X)_Y/ (D_n (X)_Y+D_n(X,\rho)_Y)$,  and we denote by $\partial_n^{\rm SQ}$ the induced boundary homomorphism $\partial_n^{\rm SQ}: C_n^{\rm SQ} (X,\rho)_Y \to C_{n-1}^{\rm SQ} (X,\rho)_Y$.  
The quotient chain complex $C_*^{\rm SQ} (X,\rho)_Y=\{ C_n^{\rm SQ} (X,\rho)_Y, \partial_n^{\rm SQ} \}_{n\in \mathbb Z}$ leads to the {\it symmetric quandle homology group} of $(X,\rho)$ and $Y$ defined by $H_n^{\rm SQ} (X,\rho)_Y= {\rm Ker} \partial_n^{\rm SQ}/ {\rm Im} \partial_{n+1}^{\rm SQ}$.

For an abelian group $A$, we define the chain complex $C_*^{\rm SQ} (X,\rho ;A)_Y =\{ C_n^{\rm SQ} (X,\rho)_Y\otimes A, \partial_n^{\rm SQ} \otimes {\rm id} \}_{n\in \mathbb Z}$. 
The homology groups are denoted by $H_*^{\rm SQ} (X,\rho; A)_Y$.
Note that when $A=\mathbb Z$, $H_*^{\rm SQ} (X,\rho; \mathbb Z)_Y$ represents $H_*^{\rm SQ} (X,\rho)_Y$.
We also define the cochain group $C_{\rm SQ}^n (X,\rho ; A)_Y$ and the coboundary homomorphism $\delta_{\rm SQ}^n : C_{\rm SQ}^n (X,\rho; A)_Y \to C_{\rm SQ}^{n+1} (X,\rho; A)_Y$  
by $C_{\rm SQ}^n (X,\rho; A)_Y = {\rm Hom}(C_n^{\rm SQ}(X,\rho)_Y, A)$ and $\delta_{\rm SQ}^{n} (\theta)=  \theta \circ \partial ^{\rm SQ}_{n+1}$, respectively. 
The {\it symmetric quandle cohomology group with coefficients in $A$} is defined  
by $H_{\rm SQ}^n (X,\rho; A)_Y={\rm Ker} \delta_{\rm SQ}^n/{\rm Im} \delta_{\rm SQ}^{n-1}$, and
we denote by $Z_{\rm SQ}^n (X,\rho; A)_Y$ and $B_{\rm SQ}^n (X,\rho; A)_Y$ the cocycle group ${\rm Ker} \delta_{\rm SQ}^n$ and the coboundary group ${\rm Im} \delta_{\rm SQ}^{n-1}$, respectively.

\subsection{Symmetric quandle colorings} \label{subsec:symmetricquandlecoloring}
Let $(X,\rho,*)$ be a symmetric quandle and $D$ an (unoriented) link diagram. 
A {\it semi-arc} of $D$ is a connected component after removing the regular neighborhoods of the crossings of $D$. 

An {\it $(X,\rho)$-coloring} of $D$ is the equivalence class of an assignment of a normal orientation and an element of $ X $ to each semi-arc  satisfying the following conditions. 
\begin{itemize}
\item Suppose that two semi-arcs coming from an over-arc of $D$ at a crossing are labeled by $x_1$ and $x_2$. 
If the normal orientations are coherent, then $x_1=x_2$, otherwise $x_1=\rho(x_2)$. See the left  upper half of Figure~\ref{symmetricquandlecoloring1}.
\item Suppose that two semi-arcs $\alpha_1$ and $\alpha_2$ which are under-arcs at a crossing $\chi$ are labeled by $x_1$ and $x_2$, respectively, and suppose that one of the semi-arcs coming from the over arc, say $\alpha_3$, of $\chi$ is labeled by $x_3$. 
We assume that the normal orientation of $\alpha_3$ points from $\alpha_1$ to $\alpha_2$.
If the normal orientations of $\alpha_1$ and $\alpha_2$ are coherent, then $x_1 * x_3 =x_2$ holds, otherwise $x_1*x_3 = \rho(x_2)$ holds.
See the  left lower  half of Figure~\ref{symmetricquandlecoloring1}.
\end{itemize}
\begin{figure}[t]
 \begin{center}
\includegraphics[width=110mm]{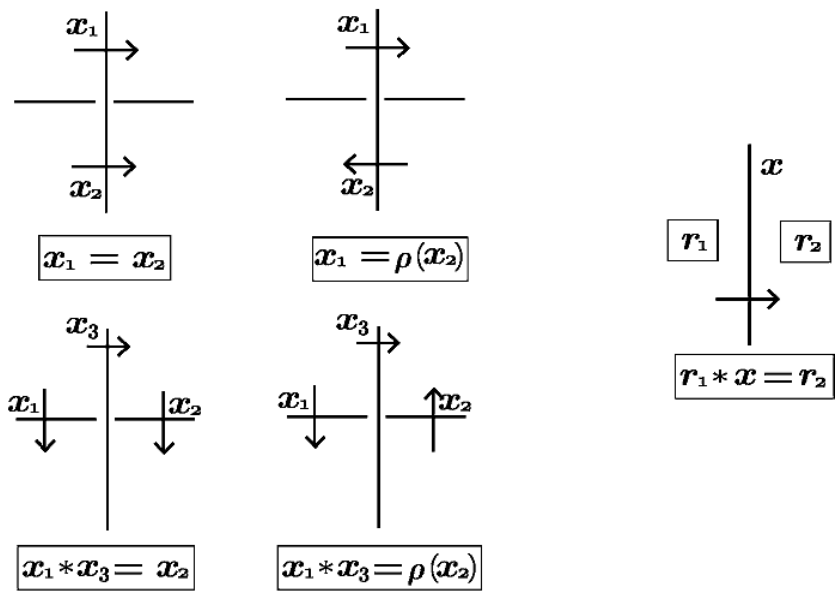}
\end{center}
\caption{The coloring conditions of symmetric quandle colorings}
\label{symmetricquandlecoloring1}
\end{figure}
Here the equivalence relation is generated by {\it basic inversions} which are operations of replacing the normal orientation of a semi-arc with the inverse one  and the element $x$, assigned to the semi-arc, with $\rho(x)$, see Figure~\ref{basicinversion}.  
\begin{figure}[t]
 \begin{center}
\includegraphics[width=50mm]{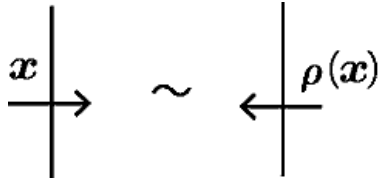}
\end{center}
\caption{A basic inversion}
\label{basicinversion}
\end{figure}
We denote by ${\rm Col}_{(X,\rho)}(D)$ the set of $(X,\rho)$-colorings of $D$. We then have the following lemma.
\begin{lemma}[\cite{Kamada2006, KamadaOshiro10}]
Let $D$ and $D'$ be diagrams of (unoriented) links.
If $D$ and $D'$ represent the same link, then there exists a one-to-one correspondence between ${\rm Col}_{(X, \rho)}(D)$ and ${\rm Col}_{(X,\rho)}(D')$. 
\end{lemma}
 This implies that $\# {\rm Col}_{(X,\rho)}(D)$, that is the cardinality of ${\rm Col}_{(X,\rho)}(D)$, is an invariant for links, and hence, we also denote it by $\# {\rm Col}_{(X,\rho)}(L)$ and call it the {\it symmetric quandle coloring number} of $L$ when $X$ is finite, where $L$ is a link which $D$ represents.

Let $Y$ be an $(X,\rho)$-set. An {\it $(X,\rho)_Y$-coloring} of $D$ is an $(X,\rho)$-coloring of $D$ with an assignment of an element of $Y$ to each complementary region of $D$ satisfying the following condition.  
\begin{itemize}
\item Suppose that two complementary regions $\gamma_1$ and $\gamma_2$ are adjacent over a semi-arc $\alpha$, and  $\gamma_1$, $\gamma_2$ and $\alpha$ are labeled by $r_1$, $r_2$ and $x$, respectively. We assume that the normal orientation of $\alpha$ points from $\gamma_1$ to $\gamma_2$. Then $r_1*x=r_2$ holds.  See the right of Figure~\ref{symmetricquandlecoloring1}.
\end{itemize}
We denote by ${\rm Col}_{(X,\rho)_Y}(D)$ the set of $(X,\rho)_Y$-colorings of $D$. Then we have the following lemma.
\begin{lemma}[\cite{Kamada2006, KamadaOshiro10}]
Let $D$ and $D'$ be diagrams of (unoriented) links.
If $D$ and $D'$ represent the same link, then there exists a one-to-one correspondence between ${\rm Col}_{(X,\rho)_Y}(D)$ and ${\rm Col}_{(X,\rho)_Y}(D')$. 
\end{lemma}
 This implies that $\# {\rm Col}_{(X,\rho)_Y}(D)$, that is the cardinality of ${\rm Col}_{(X,\rho)_Y}(D)$,  is an invariant for links.

\subsection{Symmetric quandle cocycle invariants of (unoriented) links}\label{Symmetric quandle cocycle invariants of (unoriented) links}

Let $(X,\rho,*)$ be a symmetric quandle and $Y$ an $(X,\rho)$-set.
Let $D$ be a diagram of an (unoriented) link $L$ and $C$ an $(X,\rho)_Y$-coloring of $D$. 
For each crossing $\chi$ of $D$, we extract a weight $w_\chi$ as follows.
For a crossing $\chi$ of $D$, there are four complementary regions of $D$ around $\chi$. We choose one of the regions, say $\gamma$, and call it the {\it specified region} for $\chi$.
Let $\alpha_1$ and $\alpha_2$ denote the under-semi-arc and the over-semi-arc of $\chi$, respectively,  which faces the specified region $\gamma$. 
By basic inversions, we may assume that the normal orientations $n_1$ and $n_2$ of $\alpha_1$ and $\alpha_2$, respectively, point from $\gamma$.
Let $x_1$, $x_2$ and $r$ be the labels of $\alpha_1$, $\alpha_2$ and $\gamma$, respectively.
Then the {\it weight} of the crossing $\chi$ is  $w_{\chi}=\varepsilon (r, x_1,x_2)$, where  if the pair $(n_1, n_2)$ of the normal orientations $n_1$ and $n_2$ assigned to $\alpha_1$ and $\alpha_2$, respectively, matches the right-handed orientation of $\mathbb R^2$, then  $\varepsilon=+$, otherwise $\varepsilon=-$, see Figure~\ref{weight-symmetricquandle}.   
\begin{figure}[t]
 \begin{center}
\includegraphics[width=90mm]{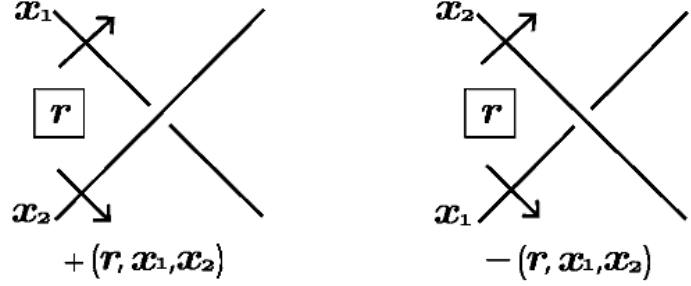}
\end{center}
\caption{The weight of a crossing}
\label{weight-symmetricquandle}
\end{figure}
Let $W^{(X,\rho)_Y}(D,C)$ denote the sum of the weights of all the crossings, and then, $W^{(X,\rho)_Y}(D,C)$ is a $2$-cycle of $C_*^{\rm SQ}(X,\rho)_Y$. 
Define
\[
\mathcal{W}^{(X,\rho)_Y} (D) = \{ [W^{(X,\rho)_Y}(D,C)] \in H_2^{\rm SQ}(X,\rho)_Y ~|~ C\in {\rm Col}_{(X,\rho)_Y} (D)\}
\]
as a multi-set.

For an abelian group $A$, let $\theta: C_2^{\rm SQ}(X,\rho)_Y \to A$ be a symmetric quandle $2$-cocycle. We define the multi-set $\Phi_\theta^{(X,\rho)_Y}(D)$ by 
\[
\Phi_\theta^{(X,\rho)_Y}(D) = \{ \theta(W^{(X,\rho)_Y}(D,C)) \in A ~|~ C\in {\rm Col}_{(X,\rho)_Y} (D)\}.
\] 

\begin{theorem}[\cite{Kamada2006, KamadaOshiro10}]
The multi-sets $\mathcal{W}^{(X,\rho)_Y}(D)$ and $\Phi_\theta^{(X,\rho)_Y} (D)$ are invariants of the link type of $D$.
\end{theorem}
 We also denote these invariants by $\mathcal{W}^{(X,\rho)_Y} (L)$ and $\Phi_\theta^{(X,\rho)_Y}(L)$, respectively, and call them the  {\it symmetric quandle homology invariant} of $L$ by $(X,\rho)$ and $Y$ and the {\it symmetric quandle cocycle invariant} of $L$ by $(X,\rho), Y$ and $\theta$, respectively.

\section{Quandle homology groups are interpreted as symmetric quandle homology groups}\label{sec:4}
In this section, we show that for any quandle $X$ and any $X$-set $Y$, the quandle homology groups of $X$ and $Y$ is obtained as symmetric quandle homology groups.
\subsection{The symmetric double of a quandle}
Let $(X,*)$ be a quandle. 
We set two copies $X^{+1}$ and $X^{-1}$ of $X$ by 
\[
X^{\varepsilon} = \{ x^{\varepsilon} ~|~ x\in X\} ~~~~~(\epsilon \in \{\pm 1\})
\]
and let  $D_X=X^{+1} \sqcup X^{-1}$.
For $x^{\varepsilon}\in D_X$, we call $\varepsilon$ the {\it sign} of $x^{\varepsilon}$. 
Define a binary operation $\star: D_X \times D_X \to D_X$ by 
\[
x^{\varepsilon} \star y^\delta= (x*y^{\delta})^{\varepsilon} ~~~\mbox{ for $x,y \in X$ and $\varepsilon, \delta\in \{\pm1\}$.}
\]

\begin{lemma}
The set $D_X$ is a quandle with the operation $\star: D_X \times D_X \to D_X$.
\end{lemma}
\begin{proof}
{\rm (Q1)} For any $x^\varepsilon \in D_X$, we have $x^\varepsilon \star x^\varepsilon = (x*x^\varepsilon)^\varepsilon = x^\varepsilon$.
{\rm (Q2)} For any $x^\varepsilon , y^\delta \in D_X$, we set $z=x*y^{-\delta}$ and $\zeta=\varepsilon$, and then we have $z^{\zeta} \star y^\delta = (z*y^{\delta})^\zeta = x^{\varepsilon}$. 
If we have two elements $z^{\zeta}, w^{\eta} \in D_X$ such that $z^{\zeta}\star y^\delta = x^{\varepsilon}$ and $w^{\eta} \star y^\delta = x^{\varepsilon}$, then  $(z*y^{\delta})^\zeta= (w* y^\delta)^\eta$ holds, which implies $\zeta=\eta$ and $z=w$. 
{\rm (Q3)} For any  $x^\varepsilon , y^\delta, z^\zeta \in D_X$, by Lemma~\ref{lem:1}, we have $(x^\varepsilon \star y^\delta) \star z^\zeta=(x*y^\delta)^\varepsilon \star z^\zeta = \big( (x*y^\delta) * z^\zeta \big)^\varepsilon=  \big( (x*z^{\zeta})*(y*z^\zeta)^\delta \big)^\varepsilon = (x*z^{\zeta})^\varepsilon\star(y*z^\zeta)^\delta =(x^\varepsilon \star z^{\zeta})\star(y^\delta\star z^\zeta) $.
\end{proof}
 We call the quandle $(D_X,\star)$ the {\it double}  of $X$.  
\begin{lemma}\label{lem:X+1}
The subset $X^{+1}$ (resp. $X^{-1}$) of $D_X$ forms a subquandle of $D_X$. 
Moreover, we have $X^{+1} \cong X$ (resp. $X^{-1}\cong X$).  
\end{lemma}
\begin{proof}
It is easily seen that the quandle operation $\star$ on $X^{+1}$ preserves 
the sign $+1$, and thus, $(X^{+1},\star)$ is a subquandle of $D_X$. 
The isomorphism $X^{+1} \to X$ is given by  $x^{+1} \mapsto x$.
\end{proof}

Define a map $\rpm: D_X \to D_X$ by $\rho(x^\varepsilon) = x^{-\varepsilon}$.
\begin{lemma}
The map $\rpm: D_X \to D_X$ is a good involution of the double $(D_X,\star)$ of $X$.
\end{lemma}
\begin{proof}
(GI1) For any $x^\varepsilon , y^\delta \in D_X$, we have $\rpm(x^\varepsilon \star y^\delta)=\rpm((x * y^\delta)^\varepsilon)=(x * y^\delta)^{- \varepsilon}=x ^{- \varepsilon} \star y^\delta=\rpm(x ^{\varepsilon}) \star y^\delta$. 
(GI2)  For any $x^\varepsilon , y^\delta \in D_X$, we have $(x^\varepsilon \star \rpm(y^\delta) ) \star y^\delta =(x^\varepsilon \star y^{-\delta}) \star y^\delta =\big( (x*y^{-\delta})*y^\delta\big )^\varepsilon = x^\varepsilon$, and hence, $x^\varepsilon \star \rpm(y^\delta)=x^\varepsilon \star (y^\delta)^{-1}$ holds.
\end{proof}
 We call the symmetric quandle $(D_X, \rpm)$ the {\it symmetric double} of $X$.

\begin{example}
Let $n$ be an odd integer and  $X$ the dihedral quandle $R_n$. 
The double $D_X$ of $X$ is isomorphic to the dihedral quandle $R_{2n}$, where the isomorphism $\phi : D_X \to R_{2n}$ is defined by 
$$\phi(x^{\varepsilon})=
\left\{
\begin{array}{ll}
2x & \mbox{if $\varepsilon = +1$,}\\
2x+n & \mbox{if $\varepsilon = -1$.}
\end{array}
\right.
 $$  
Set $\rho : R_{2n} \to R_{2n}$ by $\rho (x) = x+n$ for $x\in R_{2n}$, where we note that this $\rho$ is a good involution of $R_{2n}$. 
Then the symmetric double $(D_X,\rpm)$ of $X$  is isomorphic to the symmetric quandle  $(R_{2n}, \rho)$.  
\end{example}

Let $Y$ be an $X$-set. 
For the associated groups $G_X$ and $G_{(D_X,\rpm)}$, we define the group homomorphism $\pi : G_{(D_X,\rpm)} \to G_X$ by $\pi(x^\varepsilon) = x^\varepsilon$ for $x^{\varepsilon}\in D_X$. 
Note that $x^{+1}$ represents $x$ in $G_X$ and $x^{-1}$ represents the inverse of $x$ in $G_X$.
The next property can be proven by an easy verification as in Lemmas~\ref{lem:1} and \ref{lem:2}.
\begin{lemma}\label{lem:Y}
$Y$ is a $(D_X, \rpm)$-set by the right action 
$$\star: Y \times G_{(D_X,\rpm)} \to Y; ~(a,g) \mapsto a \star g = a * \pi(g).$$
\end{lemma}
\begin{remark}
For $x^\varepsilon \in D_X$, $(x^{\varepsilon})^{-1}=\rpm(x^{\varepsilon})=x^{-\varepsilon} \in G_{(D_X,\rpm)}$ and $(x^{\varepsilon})^{+1}=x^{\varepsilon}\in G_{(D_X,\rpm)}$ hold, which imply that for $x^{\varepsilon} \in D_X$ and $\delta \in \{\pm 1\}$, we may apply the multiplication  $(x^{\varepsilon})^\delta = x^{\varepsilon \delta}$ to the signs.
\end{remark}

\subsection{The homology groups of  a quandle and its symmetric double}
In this subsection, we show that the homology group of a given quandle and that of the symmetric double  are isomorphic. This property is very important for applying symmetric quandle cocycle invariants for oriented (or unoriented orientable) links and saying that symmetric quandle cocycle invariants are a generalization of quandle cocycle invariants.

For simplicity, in this subsection, we omit parenthesis of operations or actions such as 
$x_1 * x_2^{\varepsilon_2} * x_3^{\varepsilon_3}*\cdots *  x_n^{\varepsilon_n}$
 means 
$(\cdots ((x_1 * x_2^{\varepsilon_2}) * x_3^{\varepsilon_3})*\cdots ) *  x_n^{\varepsilon_n}$, and 
$x_1 \star x_2^{\varepsilon_2} \star x_3^{\varepsilon_3}\star \cdots \star  x_n^{\varepsilon_n}$
 means 
$(\cdots ((x_1 \star x_2^{\varepsilon_2}) \star x_3^{\varepsilon_3}) \star \cdots ) \star  x_n^{\varepsilon_n}$.

Let $(X,*)$ be a quandle, $(D_X,\rpm,\star)$ the symmetric double of $X$, and  $Y$ an $X$-set. 
We note that by Lemma~\ref{lem:Y}, $Y$ is regarded as a $(D_X, \rpm)$-set.

For each $(r, x_1^{\varepsilon_1}, \ldots , x_n^{\varepsilon_n} )\in Y \times D_X^n$, 
apply the following process (Step 0)-(Step 2), say the {\it canonicalization}, while $\{\delta_1,  \ldots, \delta _n\} \not = \{+1\}$.  
\begin{itemize} 
\item(Step 0) Set the  singed $(n+1)$-tuple $\delta (s,y_1^{\delta_1}, \ldots , y_n^{\delta_n})\in \{\pm 1\} \times Y \times D_X^n$ by $(s,y_1^{\delta_1}, \ldots , y_n^{\delta_n}):=(r,x_1^{\varepsilon_1}, \ldots , x_n^{\varepsilon_n})$, $\delta:=+1$ and $i:=n$. Go to (Step 1).
\item(Step 1) If the $i$th sign $\delta_i$ is $-1$,
replace 
$\delta (s,y_1^{\delta_1}, \ldots , y_n^{\delta_n})$ with
\begin{align}
-\delta(s*  y_i^{\delta_i} ,(y_1* y_i^{\delta_i})^{\delta_1}, \ldots, (y_{i-1}* y_i^{\delta_i})^{\delta_{i-1}} , y_i^{-\delta_i} , y_{i+1}^{\delta_{i+1}}, \ldots , y_n^{\delta_n}),  \label{eq:1}
\end{align}
and then,  denote by $\delta (s,y_1^{\delta_1}, \ldots , y_n^{\delta_n})$ the replaced signed $(n+1)$-tuple (\ref{eq:1}).  Go to (Step 2).
 \item(Step 2) Set $i:=i-1$. Go to (Step 1).
\end{itemize}
We note that applying the canonicalization should be finished once  
we have $\delta (s,y_1^{\delta_1}, \ldots , y_n^{\delta_n})$ with $\{\delta_1,  \ldots, \delta _n\} = \{+1\}$.
Then we have an element $\delta (s, y_1^{+1}, \ldots , y_n^{+1} )\in \{\pm 1\} \times Y\times (X^{+1})^n$,
where we note that the transformation in (Step 1) comes from the generators of the subgroup $D_n(D_X, \rpm)_Y$ of $C_n(D_X)_Y$.
This implies that any $(n+1)$-tuple  $(r, x_1^{\varepsilon_1}, \ldots , x_n^{\varepsilon_n} )$ in $C_n^{\rm SQ} (D_X,\rpm)_Y$  is represented by a single term $\delta (s, y_1^{+}, \ldots , y_n^{+} )$, that is, 
$$(r, x_1^{\varepsilon_1}, \ldots , x_n^{\varepsilon_n} )=\delta (s, y_1^{+1}, \ldots , y_n^{+1} ) \in C_n^{\rm SQ}(D_X,\rpm)_Y.$$
We call the resultant element $\delta (s, y_1^{+1}, \ldots , y_n^{+1} ) \mbox{ in } C_n^{\rm SQ} (D_X,\rpm)_Y$ the {\it canonical form} of $(r, x_1^{\varepsilon_1}, \ldots , x_n^{\varepsilon_n} )$.
\begin{example}
For $(r, x_1^{-1}, x_2^{+1}, x_3^{-1}) \in Y \times D_X^3$, we apply the canonicalization until all the signs of the last three components become $+1$.  Since  we have the transformations
\[\begin{array}{lcl}
&&\hspace{-1.2cm}(r, x_1^{-1}, x_2^{+1}, x_3^{-1})\\[3pt]
& \longrightarrow& -(r*x_3^{-1}, (x_1*x_3^{-1})^{-1}, (x_2*x_3^{-1})^{+1}, x_3^{+1}) \\[3pt]
&\longrightarrow & 
+((r*x_3^{-1})*(x_1*x_3^{-1})^{-1}, (x_1*x_3^{-1})^{+1}, (x_2*x_3^{-1})^{+1}, x_3^{+1})\\[3pt]
&=&+(r*x_1^{-1}*x_3^{-1}, (x_1*x_3^{-1})^{+1}, (x_2*x_3^{-1})^{+1}, x_3^{+1}),
\end{array}
\]
the canonical form of $(r, x_1^{-1}, x_2^{+1}, x_3^{-1})$ is $+(r*x_1^{-1}*x_3^{-1}, (x_1*x_3^{-1})^{+1}, (x_2*x_3^{-1})^{+1}, x_3^{+1})$.
\end{example}

Define a homomorphism  $T_n^{{D_X \to X} }:  C_n^{\rm SQ} (D_X, \rpm )_Y \to C_n^{\rm Q} (X)_Y$ by 
\begin{align*}
T_n^{D_X \to X}(r, x_1^{\varepsilon_1}, \ldots , x_n^{\varepsilon_n})= \delta  (s, y_1, \ldots , y_n )
\end{align*}
for $n\geq 1$ and $T_n^{D_X \to X}=0$ otherwise, where  $\delta  (s, y_1^{+1}, \ldots , y_n^{+1} )$ is the canonical form of  $(r, x_1^{\varepsilon_1}, \ldots , x_n^{\varepsilon_n})$. Define   a homomorphism $T_n^{X \to D_X}:  C_n^{\rm Q} (X)_Y \to C_n^{\rm SQ} (D_X, \rpm )_Y $ by 
\[
T_n^{X \to D_X}(r, x_1, \ldots , x_n)=(r, x_1^{+1}, \ldots  , x_n^{+1})
\]
for $n\geq 1$ and $T_n^{X \to D_X}=0$ otherwise.
\begin{lemma}\label{lem:bijective}
$T_n^{D_X \to X}$ and $T_n^{X \to D_X}$ are inverses of each other, and thus, they are isomorphisms.
\end{lemma}  
\begin{proof}
First, we will check the well-definedness of  $T_n^{D_X \to X}$.
We temporarily regard this map as $T_n^{D_X \to X}:  C_n (D_X)_Y \to C_n (X)_Y$.
We may show that for any $(r, x_1^{\varepsilon_1}, \ldots , x_n^{\varepsilon_n})\in Y \times D_X^n$ and $i\in \{1,\ldots , n\}$,
\begin{align}
&T_n^{D_X \to X}\Big( (r, x_1^{\varepsilon_1}, \ldots , x_n^{\varepsilon_n})\notag \\ 
&\hspace{1em}+(r\star x_{i}^{\varepsilon_i}, x_1^{\varepsilon_1}\star x_{i}^{\varepsilon_i}, \ldots ,x_{i-1}^{\varepsilon_{i-1}}\star x_{i}^{\varepsilon_i}, \rho(x_{i}^{\varepsilon_i}), x_{i+1}^{\varepsilon_{i+1}},\ldots,  x_n^{\varepsilon_n})\Big) =0, \label{equation-of-rho}\\
&T_n^{D_X \to X}(r, x_1^{\varepsilon_1}, \ldots , x_i^{\varepsilon_i} , x_{i+1}^{\varepsilon_{i+1}}=x_i^{\varepsilon_i}, \ldots , x_n^{\varepsilon_n}) \in D_n(X)_Y. \label{equation-of-D}
\end{align}
Let $j_1, \ldots, j_\ell$ denote the elements   of  $\{i+1, \ldots , n\}$ such that $j_1< j_2<\cdots < j_\ell$ and $\varepsilon_{j_1}=\varepsilon_{j_2}= \cdots = \varepsilon_{j_\ell}=-1$. Assume $\varepsilon_i=+1$.
For the $(n+1)$-tuple $ (r, x_1^{\varepsilon_1}, \ldots , x_i^{+1}, \ldots , x_n^{\varepsilon_n})$, apply the canonicalization while $\{\delta_{j_1}, \ldots , \delta_{j_\ell}\}\not = \{+1\}$. 
We note that applying the canonicalization should be finished once  
we have $\delta (s,y_1^{\delta_1}, \ldots , y_n^{\delta_n})$ with $\{\delta_{j_1}, \ldots , \delta_{j_\ell}\}= \{+1\}$.
We then have the transformation
\[
\begin{array}{l}
(r, x_1^{\varepsilon_1}, \ldots ,  x_i^{+1}, \ldots , x_n^{\varepsilon_n}) \longrightarrow
(-1)^{\ell} (t, z_1^{\varepsilon_1}, \ldots , z_{i-1}^{\varepsilon_{i-1}}, z_i^{+1}, z_{i+1}^{+1},\ldots , z_{n}^{+1}),
\end{array}
\]     
which implies that 
\[
\begin{array}{l}
T_n^{D_X \to X}(r, x_1^{\varepsilon_1}, \ldots,  x_i^{+1}, \ldots  , x_n^{\varepsilon_n})\\[3pt]
=(-1)^{\ell} T_n^{D_X \to X} (t, z_1^{\varepsilon_1}, \ldots , z_{i-1}^{\varepsilon_{i-1}}, z_i^{+1}, z_{i+1}^{+1},\ldots , z_{n}^{+1}),
\end{array}
\]     
where $(-1)^{\ell} (t, z_1^{\varepsilon_1}, \ldots , z_{i-1}^{\varepsilon_{i-1}}, z_i^{+1}, z_{i+1}^{+1},\ldots , z_{n}^{+1})$ is the resultant $\delta (s,y_1^{\delta_1}, \ldots , y_n^{\delta_n})$ by the above canonicalization, that is, 
$t=r * x_{j_1}^{-1} * \cdots * x_{j_\ell}^{-1}$ and 
\[ 
z_k=\left\{
\begin{array}{cl}
x_k * x_{j_1}^{-1} * \cdots * x_{j_\ell}^{-1} & ( 1\leq k < j_1),\\
$\vdots $ &\hspace{1cm}$\vdots $\\
x_k * x_{j_{\kappa+1}}^{-1} * \cdots * x_{j_{\ell}}^{-1} & ( j_\kappa \leq k < j_{\kappa+1}, \kappa \in \{1,\ldots, \ell-1 \}),\\
$\vdots $ &\hspace{1cm}$\vdots $\\
x_k  & ( j_\ell \leq k \leq n ).\\
\end{array}
\right.
\]
For the $(n+1)$-tuple $(r\star x_{i}^{\varepsilon_i}, x_1^{\varepsilon_1}\star x_{i}^{\varepsilon_i}, \ldots ,x_{i-1}^{\varepsilon_{i-1}}\star x_{i}^{\varepsilon_i}, \rho(x_{i}^{\varepsilon_i}), x_{i+1}^{\varepsilon_{i+1}},\ldots,  x_n^{\varepsilon_n})$, that is $(r* x_{i}, (x_1* x_{i})^{\varepsilon_1}, \ldots , (x_{i-1}* x_{i})^{\varepsilon_{i-1}}, x_{i}^{-1}, x_{i+1}^{\varepsilon_{i+1}},\ldots,  x_n^{\varepsilon_n})$, 
apply the canonicalization while $\{ \delta_i, \delta_{j_1}, \ldots , \delta_{j_\ell}\}\not = \{+1\}$.
We note that applying the canonicalization should be finished once  
we have $\delta (s,y_1^{\delta_1}, \ldots , y_n^{\delta_n})$ with $\{\delta_i,  \delta_{j_1}, \ldots , \delta_{j_\ell}\}= \{+1\}$.
We then have the transformation 
\[
\begin{array}{l}
(r* x_{i}, (x_1* x_{i})^{\varepsilon_1}, \ldots , (x_{i-1}* x_{i})^{\varepsilon_{i-1}}, x_{i}^{-1}, x_{i+1}^{\varepsilon_{i+1}},\ldots,  x_n^{\varepsilon_n})\\[3pt]
\longrightarrow (-1)^{\ell+1}  (t', (z_1')^{\varepsilon_1}, \ldots , (z_{i-1}')^{\varepsilon_{i-1}}, (z_i')^{+1}, (z_{i+1}')^{+1},\ldots , (z_{n}')^{+1}),
\end{array}
\]
which implies that 
\[
\begin{array}{l}
T_n^{D_X \to X}(r* x_{i}, (x_1* x_{i})^{\varepsilon_1}, \ldots , (x_{i-1}* x_{i})^{\varepsilon_{i-1}}, x_{i}^{-1}, x_{i+1}^{\varepsilon_{i+1}},\ldots,  x_n^{\varepsilon_n})\\[3pt]
=(-1)^{\ell+1} T_n^{D_X \to X} (t', (z_1')^{\varepsilon_1}, \ldots , (z_{i-1}')^{\varepsilon_{i-1}}, (z_i')^{+1}, (z_{i+1}')^{+1},\ldots , (z_{n}')^{+1}),
\end{array}
\]     
where 
$(-1)^{\ell+1}  (t', (z_1')^{\varepsilon_1}, \ldots , (z_{i-1}')^{\varepsilon_{i-1}}, (z_i')^{+1}, (z_{i+1}')^{+1},\ldots , (z_{n}')^{+1})$ is the resultant $\delta (s,y_1^{\delta_1}, \ldots , y_n^{\delta_n})$ by the above canonicalization, and 
by Lemmas~\ref{lem:1} and \ref{lem:2}, we have
\[
\begin{array}{ll}
t'&=((r*x_i) * x_{j_1}^{-1} * \cdots * x_{j_\ell}^{-1})* (x_i* x_{j_1}^{-1} * \cdots * x_{j_\ell}^{-1})^{-1}\\[3pt]
&=((r*x_i) *x_i^{-1}) * x_{j_1}^{-1} * \cdots * x_{j_\ell}^{-1}\\[3pt]
&=r * x_{j_1}^{-1} * \cdots * x_{j_\ell}^{-1}\\[3pt]
&= t, 
\end{array}
\] 
\[
\begin{array}{ll}
z_k'&=((x_k*x_i) * x_{j_1}^{-1} * \cdots * x_{j_\ell}^{-1})* (x_i* x_{j_1}^{-1} * \cdots * x_{j_\ell}^{-1})^{-1}\\[3pt]
&=((x_k*x_i) *x_i^{-1}) * x_{j_1}^{-1} * \cdots * x_{j_\ell}^{-1}\\[3pt]
&=x_k* x_{j_1}^{-1} * \cdots * x_{j_\ell}^{-1}\\[3pt]
&= z_k 
\end{array}
\]
 for $k\in \{1,\ldots , i-1\}$, 
\[
\begin{array}{ll}
(z_i')^{+1} &= \rho( (x_i *x_{j_1}^{-1} * \cdots * x_{j_\ell}^{-1})^{-1})\\[3pt]
&= (x_i *x_{j_1}^{-1} * \cdots * x_{j_\ell}^{-1})^{+1}\\[3pt]
&= z_i^{+1},
\end{array}
\]
and $z_k' = z_k$ for $k\in\{i+1, \ldots, n\}$. 
Thus we have
\[
\begin{array}{l}
T_n^{D_X \to X}(r\star x_{i}^{+1}, x_1^{\varepsilon_1}\star x_{i}^{+1}, \ldots ,x_{i-1}^{\varepsilon_{i-1}}\star x_{i}^{+1}, \rho(x_{i}^{+1}),  x_{i+1}^{\varepsilon_{i+1}},\ldots,  x_n^{\varepsilon_n})\\[3pt]
=T_n^{D_X \to X}(r* x_{i}, (x_1* x_{i})^{\varepsilon_1}, \ldots , (x_{i-1}* x_{i})^{\varepsilon_{i-1}}, x_{i}^{-1}, x_{i+1}^{\varepsilon_{i+1}},\ldots,  x_n^{\varepsilon_n})\\[3pt]
=(-1)^{\ell+1} T_n^{D_X \to X} (t', (z_1')^{\varepsilon_1}, \ldots , (z_{i-1}')^{\varepsilon_{i-1}}, (z_i')^{+1}, (z_{i+1}')^{+1},\ldots , (z_{n}')^{+1})\\[3pt]
=(-1)^{\ell+1} T_n^{D_X \to X} (t, z_1^{\varepsilon_1}, \ldots , z_{i-1}^{\varepsilon_{i-1}}, z_i^{+1}, z_{i+1}^{+1},\ldots , z_{n}^{+1})\\[3pt]
=- T_n^{D_X \to X} (r, x_1^{\varepsilon_1}, \ldots ,  x_i^{+1}, \ldots ,x_n^{\varepsilon_n}),
\end{array}
\]
and hence, the equation (\ref{equation-of-rho}) holds.
In the case that $\varepsilon_i=-1$, the equation (\ref{equation-of-rho}) also holds, and we leave the proof to the reader. Thus for any element $z\in D_n (D_X,\rpm)_Y$, $T_n^{D_X \to X} (z)=0$ holds.
For any $(r, x_1^{\varepsilon_1}, \ldots , x_i^{\varepsilon_i} , x_{i+1}^{\varepsilon_{i+1}}, \ldots , x_n^{\varepsilon_n}) \in Y \times D_X^n$ 
such that $x_i^{\varepsilon_i}= x_{i+1}^{\varepsilon_{i+1}}$ for some $i\in \{1,\ldots , n-1\}$, 
the canonical form $\delta(s, y_1^{+1}, \ldots , y_n^{+1})\in \{\pm 1\} \times Y \times D_X^n$ of $(r, x_1^{\varepsilon_1}, \ldots , x_i^{\varepsilon_i} , x_{i+1}^{\varepsilon_{i+1}}, \ldots , x_n^{\varepsilon_n})$ satisfies that 
$y_i^{+1}=y_{i+1}^{+1}$, and hence, we have 
\[
\begin{array}{l}
T_n^{D_X \to X}(r, x_1^{\varepsilon_1}, \ldots , x_i^{\varepsilon_i} , x_{i+1}^{\varepsilon_{i+1}}=x_i^{\varepsilon_i}, \ldots , x_n^{\varepsilon_n})\\[3pt]
=\delta (s, y_1^{+1} ,\ldots , y_i^{+1}, y_{i+1}^{+1}=y_i^{+1}, \ldots , y_n^{+1} )\in D_n(X)_Y\\[3pt]
=\delta (s, y_1 ,\ldots , y_i, y_{i+1}=y_i, \ldots , y_n )\in D_n(X)_Y,
\end{array}
\] 
and hence, the property (3) holds. Thus for any element $z\in D_n (D_X)_Y$, $T_n^{D_X \to X} (z)\in D_n(X)_Y$ holds. 
As a consequence, we see that the map $T_n^{D_X \to X} : C_{n}^{\rm SQ}(D_X,\rpm)_Y \to C_n^{\rm Q} (X)_Y$ is well-defined.


We can easily  check the well-definedness of  $T_n^{X \to D_X}$ as follows. 
We temporarily regard this map as $T_n^{X \to D_X}: C_n(X)_Y \to C_n(D_X)_Y$.
For any $(r, x_1, \ldots  , x_n) \in Y \times X^n$ such that $x_i=x_{i+1}$ for some $i\in \{1,\ldots , n-1\}$, we have $x_i^{+1} = x_{i+1}^{+1}\in D_X$, and hence, 
$$
\begin{array}{l}
T_n^{X \to D_X}(r, x_1, \ldots, x_i, x_{i+1}=x_i, \ldots  , x_n)\\[3pt]
=(r,  x_1^{+1}, \ldots, x_i^{+1}, x_{i+1}^{+1}= x_i^{+1}, \ldots  , x_n^{+1})\in D_n(D_X)_Y.
\end{array}
$$ 
Thus, for any $z\in D_n(X)_Y$, $T_n^{X \to D_X}(z)\in D_n(D_X)_Y$ holds, which implies that $T_n^{X \to D_X}:  C_n^{\rm Q} (X)_Y \to C_n^{\rm SQ} (D_X, \rpm )_Y$ is well-defined. 

Finally, the property $T_n^{X \to D_X}\circ T_n^{D_X \to X} = {\rm id}_{C_n^{\rm SQ} (D_X,\rpm)_Y}$ follows from 
\begin{align*}
&T_n^{X \to D_X}\circ T_n^{D_X \to X} (r, x_1^{\varepsilon_1}, \ldots ,x_n^{\varepsilon_n}) \\
&=\delta  T_n^{X \to D_X}\circ T_n^{D_X \to X} (s, y_1^{+1}, \ldots , y_n^{+1}) \\
&=  \delta  T_n^{X \to D_X} (s, y_1, \ldots , y_n)\\
&=\delta  (s, y_1^{+1}, \ldots , y_n^{+1})\\
&=(r, x_1^{\varepsilon_1}, \ldots ,x_n^{\varepsilon_n}),
\end{align*}
where $\delta (s, y_1^{+1}, \ldots , y_n^{+1})$ is the canonical form of $(r, x_1^{\varepsilon_1}, \ldots ,x_n^{\varepsilon_n})$. 
The property $T_n^{D_X \to X} \circ T_n^{X \to D_X} = {\rm id}_{C_n^{\rm Q} (X)_Y}$ follows from 
\begin{align*}
&T_n^{D_X \to X} \circ T_n^{X \to D_X} (r, x_1, \ldots ,x_n) \\
&=T_n^{D_X \to X}  (r, x_1^{+1}, \ldots ,x_n^{+1}) \\
&=  (r, x_1, \ldots ,x_n). 
\end{align*}
This completes the proof.
\end{proof}

\begin{remark}\label{remark4.7}
As shown in above, any $(n+1)$-tuple $(r, x_1^{\varepsilon_1}, \ldots, x_n^{\varepsilon})\in C_n^{\rm SQ} (D_X, \rpm)_Y$  is represented by a single term $\delta(s, y_1^{+1}, \ldots , y_n^{+1})$, which implies that $C_n^{\rm SQ}(D_X,\rpm)_Y$ is generated by the elements of  $Y\times (X^{+1})^n$. 
Therefore, the  isomorphism $T_n^{D_X \to X} $ is, in other words, defined by 
\[
T_n^{D_X \to X}(r, x_1^{+1}, \ldots , x_n^{+1})= (r, x_1, \ldots , x_n).
\]
\end{remark}

\begin{lemma}\label{keylemma4.8}
$T_*^{D_X \to X}=\{T_n^{D_X \to X}\}_{n\in \mathbb Z}$  and $T_*^{X \to D_X}=\{T_n^{X \to D_X}\}_{n\in \mathbb Z}$ are chain maps, and therefore by Lemma~\ref{lem:bijective}, 
they are isomorphisms between the chain complexes $C_*^{\rm Q} (X)_Y$ and $ C_*^{\rm SQ} (D_X, \rpm )_Y$.
\end{lemma}
\begin{proof}
For any $(r, x_1^{+1} , \ldots ,x_n^{+1}) \in C_n^{\rm SQ} (D_X,\rpm)_Y$,
\[
\begin{array}{l}
T_{n-1}^{D_X \to X}\circ \partial _n^{\rm SQ} (r, x_1^{+1} , \ldots ,x_n^{+1})\\
= T_{n-1}^{D_X \to X}
\Big(
\displaystyle \sum_{i=1}^n (-1)^{i}\Big\{ (r, x_1^{+1},  \ldots , x_{i-1}^{+1}, x_{i+1}^{+1}, \ldots, x_n^{+1})\\
\hspace{2cm}- (r*x_i, (x_1*x_i)^{+1}  ,\ldots , (x_{i-1}*x_i)^{+1} , x_{i+1}^{+1}, \ldots , x_n^{+1})\Big\}\Big)\\
= 
\displaystyle \sum_{i=1}^n (-1)^{i}\Big\{ (r, x_1,  \ldots , x_{i-1},  x_{i+1},\ldots, x_n)\\
\hspace{2cm}- (r*x_i, x_1*x_i,\ldots , x_{i-1}*x_i, x_{i+1}, \ldots , x_n)\Big\}\\
= 
\partial^{\rm Q}_n (r, x_1,  \ldots , x_n)=\partial^{\rm Q}_n \circ T_n^{D_X \to X} (r, x_1^{+1},  \ldots , x_n^{+1}).
\end{array}
\]
Thus $T_*^{D_X \to X}$ is a chain map from $ C_*^{\rm SQ} (D_X, \rpm )_Y$  to $C_*^{\rm Q} (X)_Y$.

Since $T_n^{X \to D_X}$ is the inverse homomorphism of $T_n^{D_X \to X}$ for any $n\in \mathbb Z$, 
\begin{align*}
\partial _n^{\rm SQ} \circ T_n^{X \to D_X}&= T_{n-1}^{X \to D_X} \circ T_{n-1}^{D_X \to X}\circ \partial _n^{\rm SQ} \circ T_n^{X \to D_X}\\
&=T_{n-1}^{X \to D_X} \circ \partial^{\rm Q}_n \circ T_n^{D_X \to X}\circ T_n^{X \to D_X}\\
&=T_{n-1}^{X \to D_X} \circ \partial^{\rm Q}_n,
\end{align*}
which implies that 
$T_*^{X \to D_X}$ is a chain map from $C_*^{\rm Q} (X)_Y$ to $ C_*^{\rm SQ} (D_X, \rpm )_Y$.
\end{proof}

Let $A$ be an abelian group. We have  the following property. 
\begin{theorem}\label{thm:homology}
For any $n\in \mathbb Z$, we have 
\[
(1)~~~H_n^{\rm Q} (X; A)_Y \cong H_n^{\rm SQ} (D_X,\rpm;A)_Y
\]
and
\[
(2)~~~H^n_{\rm Q} (X; A)_Y \cong H^n_{\rm SQ} (D_X,\rpm;A)_Y.
\]
\end{theorem}
\begin{proof}
(1) By Lemma~\ref{keylemma4.8}, since $T_*^{D_X\to X}$ is an isomorphism from $ C_*^{\rm SQ} (D_X, \rpm )_Y$  to $C_*^{\rm Q} (X)_Y$, 
$T_n^{D_X \to X}$ induces the isomorphism
\begin{align*}
(T_n^{D_X \to X})^*:  H_n^{\rm SQ} (D_X,\rpm)_Y \to  H_n^{\rm Q} (X)_Y; (T_n^{D_X \to X})^*([z])=[T_n^{D_X \to X}(z)],
\end{align*}
where the inverse homomorphism is 
\begin{align*} 
(T_n^{X \to D_X})^*: H_n^{\rm Q} (X)_Y \to H_n^{\rm SQ} (D_X,\rpm)_Y ; (T_n^{X \to D_X})^*([z])=[T_n^{X \to D_X}(z)].
\end{align*} 
Similarly, since  we have  the isomorphism $T_*^{D_X \to X} \otimes {\rm id} = \{ T_n^{D_X \to X} \otimes {\rm id}\}_{n\in \mathbb Z}$ from $C_*^{\rm SQ} (D_X, \rpm ; A)_Y$ to $C_*^{\rm Q} (X;A)_Y$, we have the induced isomorphism 
\begin{align*}
(T_n^{D_X \to X}\otimes {\rm id})^*:  H_n^{\rm SQ} (D_X,\rpm; A)_Y \to  H_n^{\rm Q} (X;A)_Y.
\end{align*}


(2) By Lemma~\ref{keylemma4.8}, since $T_*^{D_X\to X}$ is an isomorphism from $ C_*^{\rm SQ} (D_X, \rpm )_Y$  to $C_*^{\rm Q} (X)_Y$, $T_n^{D_X \to X}$ induces the isomorphism 
\[
(T^n_{X \to D_X})^{*}: H^n_{\rm Q} (X; A)_Y \to H^n_{\rm SQ} (D_X,\rpm;A)_Y; [\theta^{\rm Q}] \mapsto [\theta^{\rm Q} \circ T_n^{D_X \to X}], 
\]
the inverse homomorphism is 
\[
((T^n_{X \to D_X})^{*})^{-1}:H^n_{\rm SQ} (D_X,\rpm;A)_Y \to  H^n_{\rm Q} (X; A)_Y; [\theta^{\rm SQ}] \mapsto [\theta^{\rm SQ} \circ T_n^{X \to D_X}].
\] 
\end{proof}

\begin{corollary}\label{cor:4.9}
Let $\theta^{\rm Q}: C^{\rm Q}_n (X)_Y \to A$ and $\theta^{\rm SQ}: C^{\rm SQ}_n (D_X,\rpm)_Y \to A$ be homomorphisms such that
\begin{align*} 
\theta^{\rm Q}\circ T_n^{D_X \to X} = \theta^{\rm SQ} .
\end{align*}
Then $\theta^{\rm Q}$ is an $n$-cocycle if and only if $\theta^{\rm SQ}$ is an $n$-cocycle, that is,
\[
\theta^{\rm Q} \in Z^n_{\rm Q} (X;A)_Y  ~~\Longleftrightarrow ~~ \theta^{\rm SQ} \in Z^n_{\rm SQ} (D_X, \rpm ;A)_Y.
\]
\end{corollary}
\section{Invariants for oriented (or unoriented) links using quandles can be interpreted  by using  symmetric quandles}\label{sec:5}

In this section, we discuss oriented link invariants using quandles. Especially, we show that 
quandle coloring numbers, quandle homology invariants and quandle cocycle invariants for oriented links can be interpreted by using symmetric quandles.
\subsection{An interpretation of  quandle coloring numbers by using symmetric quandle colorings}\label{subsection:coloringnumber}

Let $X$ be a quandle and $(D_X, \rpm)$ the symmetric double of $X$.

By performing a finite number of basic inversions, we can see that any $(D_X,\rpm)$-coloring  $C$ of a diagram $D$ of an unoriented link is uniquely represented by an assignment of a normal orientation  and an element of $X^{+1}$ to each  semi-arc of $D$, see Figure~\ref{canonicalform}. 
We call the assignment the {\it canonical form} of $C$ (or of $(D,C)$).
\begin{figure}[t]
 \begin{center}
\includegraphics[width=90mm]{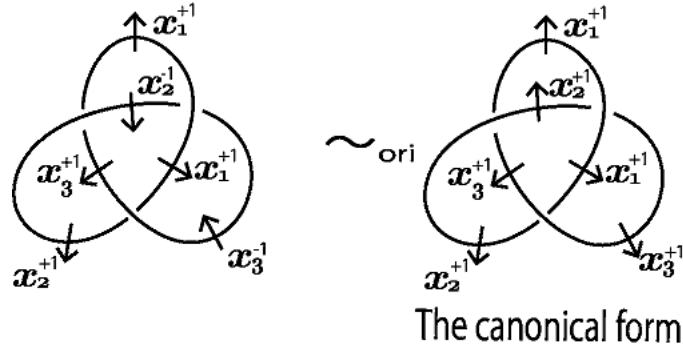}
\end{center}
\caption{The canonical form of a $(D_X,\rpm)$-coloring}
\label{canonicalform}
\end{figure}

Let $D$ be a diagram of an unoriented link and  $C$ a $(D_X,\rpm)$-coloring  of $D$.
For a semi-arc $\alpha$  of $D$ and the normal orientation $n_\alpha$  assigned to $\alpha$ for the canonical form of $C$,  
we set the orientation $o_\alpha$ of $\alpha$ so that the pair $(o_\alpha, n_\alpha)$ matches the right-handed orientation of $\mathbb R^2$.
Since for the canonical form of a $(D_X,\rpm)$-coloring of $D$, 
any two semi-arcs  of $D$ coming from the same component of $L$ have the coherent normal orientations, the above orientations $o_\alpha$ for the semi-arcs of $D$ induce an orientation $o$ of  $D$, see the upper half of  Figure~\ref{CO-map}. 
\begin{figure}[t]
 \begin{center}
\includegraphics[width=97mm]{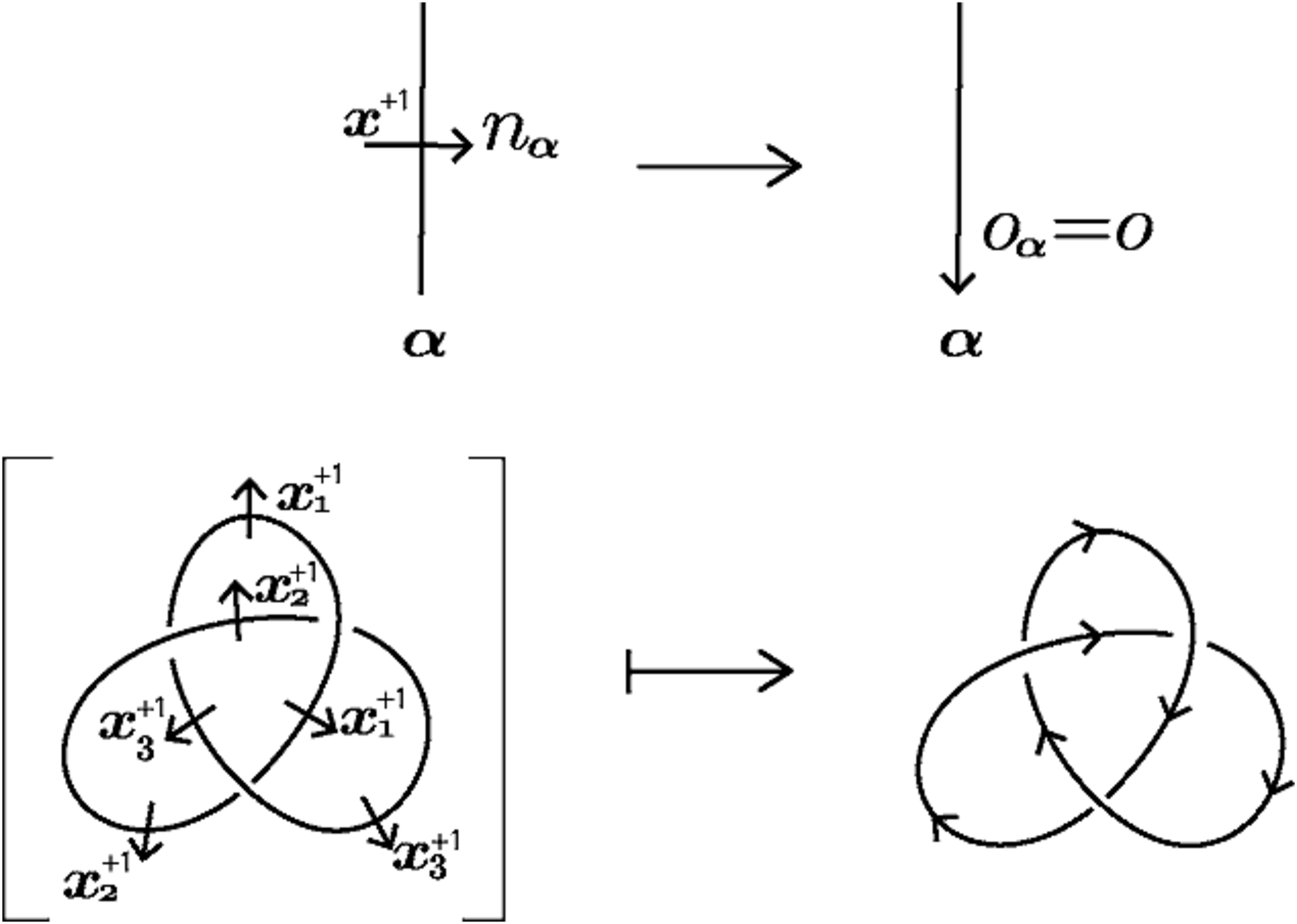}
\end{center}
\caption{The map $T_{(D,[C]_{\rm ori})\to (D,o)}$}
\label{CO-map}
\end{figure}
We call the orientation of $D$ the {\it orientation induced by the canonical form of $(D,C)$} and the diagram $(D,o)$ the {\it oriented diagram induced by the canonical form of $(D,C)$}.
Let $C$ and $C'$ be $(D_X,\rpm)$-colorings of $D$.
We say that $C$ and $C'$ are {\it orientation equivalent} $(C \sim _{\rm ori} C')$ if  the orientation induced by the canonical form of $(D, C)$ coincides with that of $(D, C')$. 
We denote by $[C]_{\rm ori}$ the orientation equivalence class of a $(D_X,\rpm)$-coloring $C$.

In this paper, ${\rm ori }(D)$ means the set of orientations of $D$.
The {\it transformation map from $(D,C)$ to $(D,o)$} is the map
\begin{align*}
&T_{(D,C)\to (D,o)}: \{D\} \times {\rm Col}_{(D_X,\rpm)} (D) \to \{D\} \times {\rm ori}(D);\\
&(D,C) \mapsto \mbox{the oriented diagram $(D,o)$ induced by the canonical form of $(D,C)$},
\end{align*}
and this induces the {\it transformation map from $(D,[C]_{\rm ori})$ to $(D,o)$}: 
\begin{align*}
&T_{(D,[C]_{\rm ori})\to (D,o)}: \{D\} \times {\rm Col}_{(D_X,\rpm)} (D)/\sim_{\rm ori} \to \{D\} \times {\rm ori}(D);\\
&(D, [C]_{\rm ori}) \mapsto \mbox{the oriented diagram $(D,o)$ induced by the canonical form of $(D,C)$},
\end{align*}
see the lower half of Figure~\ref{CO-map}.
We then have the following property.
\begin{lemma}
$T_{(D,[C]_{\rm ori})\to (D,o)}$ is bijective.
\end{lemma}
\begin{proof}
The injectivity is clear, and hence, we show  the surjectivity.
Let $o \in {\rm ori}(D)$ and $x\in X$.
For each semi-arc $\alpha$ of $(D,o)$, 
we assign the element $x^{+1}\in D_X$ and the normal orientation $n_{\alpha}$ to $\alpha$ so that the pair $(o, n_\alpha)$ matches the right-handed orientation of $\mathbb R^2$.
Then we can see that the assignment  is the canonical form of a $(D_X,\rpm)$-coloring of $D$ and the equivalence class $[C]_{\rm ori}$ of the $(D_X,\rpm)$-coloring $C$  satisfies that $T_{(D,[C]_{\rm ori})\to (D,o)}(D, [C]_{\rm ori})=(D,o)$.
\end{proof}
Any $(D_X,\rpm)$-colored diagram $(D, C)$ uniquely induces an oriented, $X$-colored diagram  $((D,o) , \bar{C})$ by setting on $D$ the orientation induced by the canonical form of $(D,C)$ and replacing, for the canonical form of $C$, the assigned element $x^{+1} \in X^{+1}$ of each (semi-)arc with $x \in X$, see the upper half of Figure~\ref{CD-map}.
\begin{figure}[t]
 \begin{center}
\includegraphics[width=97mm]{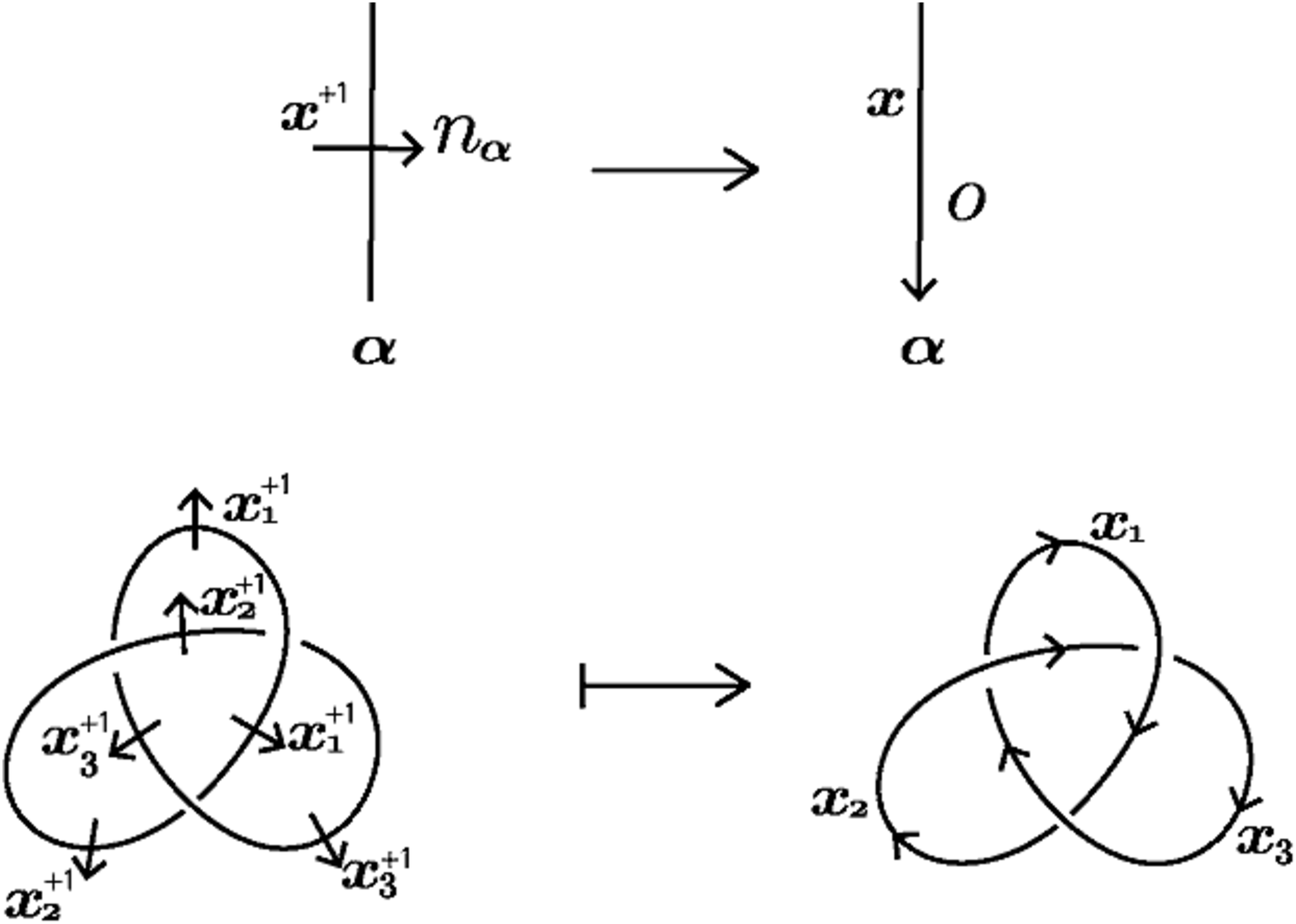}
\end{center}
\caption{The map $T_{(D,C) \to ((D,o), C)}$}
\label{CD-map}
\end{figure}
We call such $((D,o), \bar{C})$ the {\it oriented, $X$-colored diagram induced by the canonical form of $(D,C)$}.
The {\it transformation map from $(D,C)$ to $((D,o), C)$} is the map
\begin{align*}
&T_{(D,C) \to ((D,o), C)}:   \{D\} \times {\rm Col}_{(D_X,\rpm)}(D) \to  \bigcup_{o\in {\rm ori}(D)} \{(D,o)\}\times  {\rm Col}_X(D,o);\\
&(D,C) \mapsto \mbox{ the oriented, $X$-colored diagram induced by the canonical form of $(D,C)$,}
\end{align*}
 see the lower half of Figure~\ref{CD-map}.
\begin{lemma}
$T_{(D,C) \to ((D,o), C)}$ is bijective.
\end{lemma}
\begin{proof}
Any oriented, $X$-colored diagram $((D,o), \bar{C})$ uniquely induces a $(D_X,\rpm)$-colored diagram $(D, C)$ as follows.
For each arc $\alpha$ of $(D,o)$, 
replace the orientation $o$ for $\alpha$ with the normal orientations $n_\alpha$ of the semi-arcs included in $\alpha$ such that $(o, n_\alpha)$ matches the right-handed orientation of $\mathbb R^2$ and replace the assigned element $x \in X$ of $\alpha$ for $\bar{C}$ with $x^{+1}\in X^{+1}$ for the semi-arcs included in $\alpha$.  Lemma~\ref{lem:X+1} shows that the replaced assignment $C$ of the unoriented diagram $D$ represents a $(D_X,\rpm)$-coloring of $D$.  
Thus we have the inverse map of $T_{(D,C) \to ((D,o), C)}$:
\begin{align*}
&T_{((D,o), C) \to (D,C) }:  \bigcup_{o\in {\rm ori}(D)} \{(D,o)\}\times  {\rm Col}_X(D,o) \to \{D\} \times {\rm Col}_{(D_X,\rpm)}(D) ;\\
& ((D,o), C) \mapsto (D, \bar{C}) \mbox{ defined as above}.
\end{align*}
\end{proof}

The next theorem implies that quandle coloring numbers for oriented links can be interpreted by using symmetric quandle colorings.
\begin{theorem}\label{lem:correspondenceCorandori2}
Let $(D,o)$ be a diagram of an oriented link and take $[C]_{\rm ori} \in {\rm Col}_{(D_X,\rpm)} (D)/\sim_{\rm ori} $ such that 
$(D,o)  =T_{(D,[C]_{\rm ori})\to (D,o)} (D, [C]_{\rm ori} )$.
Then there exists a one-to-one correspondence between the sets ${\rm Col}_X(D,o)$ and $[C]_{\rm ori}$, and thus, we have
\[
\# {\rm Col}_X(D,o) = \# [C]_{\rm ori}.
\]
\end{theorem}
\begin{proof}
By the definition of the bijective translation map $T_{(D,C) \to ((D,o), C)}$, we have the restriction 
\begin{align*}
&T_{(D,C) \to ((D,o), C)} |_{\{D\} \times [C]_{\rm ori}} : \{D\} \times [C]_{\rm ori} \to \{(D,o)\}\times  {\rm Col}_{ X }(D,o), 
\end{align*}
which is bijective.
Clearly this map gives a one-to-one correspondence between ${\rm Col}_X(D,o)$ and $[C]_{\rm ori}$.
\end{proof}

For an unoriented link $L$, by considering all the orientations of $L$ and taking $\# {\rm Col}_X(D,o)$ for each orientation $o$, we can define an invariant of $L$. 
More precisely, for a diagram $D$ of $L$, the multi-set 
\[
\{ \# {\rm Col}_{X} (D,o) ~|~ o \in {\rm ori}(D) \} 
\]
is an invariant of $L$.  
The next property implies that quandle coloring numbers  for unoriented links can be  also interpreted by using symmetric quandle colorings.
\begin{corollary}\label{thm:coloringnumber}
Let $D$ be a diagram of an unoriented link.
As multi-sets, we have
\[
\{ \# {\rm Col}_{X} (D,o) ~|~ o  \in {\rm ori}(D) \} = \{  \# [C]_{\rm ori} ~|~ [C]_{\rm ori} \in {\rm Col}_{(D_X,\rpm)}(D)/\sim_{\rm ori}\}.
\]
\end{corollary}

\subsection{An interpretation of  quandle coloring numbers by using symmetric quandle colorings II}
In this subsection, we will see that the properties shown in Subsection~\ref{subsection:coloringnumber} are extended to the cases of $X_Y$-colorings and $(D_X, \rpm)_Y$-colorings. Note that for simplicity, we use the same notation and names as in Subsection~\ref{subsection:coloringnumber}.

Let $X$ be a quandle, $(D_X, \rpm)$ the symmetric double of $X$, and $Y$  an $X$-set. 
By Lemma~\ref{lem:Y}, we can also regard $Y$ as a $(D_X,\rpm)$-set.
Each $X_Y$-coloring $C$ of an oriented diagram $(D,o)$ is also regarded as the $X$-coloring, say also $C$, of $(D,o)$ by forgetting the assignment by $Y$.
Similarly, for each $(D_X,\rpm)_Y$-coloring $C$ of an unoriented diagram $D$ is also regarded as the $(D_X,\rpm)$-coloring, say also $C$, of $D$ by forgetting the assignment by $Y$.
Hence, the {\it canonical form} of a $(D_X,\rpm)_Y$-coloring $C$ (or of a $(D_X,\rpm)_Y$-colored diagram $(D,C)$) is defined by the assignment which is the canonical form of the $(D_X,\rpm)$-coloring $C$ when we forget the assignment by $Y$.
Moreover, as in the case of Subsection~\ref{subsection:coloringnumber}, a $(D_X,\rpm)_Y$-colored diagram $(D,C)$ induces an orientation $o$ (resp. an oriented diagram $(D,o)$) by considering the orientation (resp. the oriented diagram) induced by the canonical form of $(D,C)$.

Let $D$ be a diagram of an unoriented link.
Let $C$ and $C'$ be $(D_X,\rpm)_Y$-colorings of the diagram $D$.
We say that $C$ and $C'$ are {\it orientation equivalent} $(C \sim _{\rm ori} C')$ if  the orientation induced by the canonical form of $(D,C)$ coincides with that of $(D,C')$. 
We denote by $[C]_{\rm ori}$ the orientation equivalence class of a $(D_X,\rpm)_Y$-coloring $C$.

The {\it transformation map from $(D,C)$ to $(D,o)$} is the map
\begin{align*}
&T_{(D,C)\to (D,o)}: \{D\} \times {\rm Col}_{(D_X,\rpm)_Y} (D) \to \{D\} \times {\rm ori}(D);\\
&(D,C) \mapsto \mbox{the oriented diagram $(D,o)$ induced by the canonical form of $(D,C)$},
\end{align*}
and this induces the {\it transformation map from $(D,[C]_{\rm ori})$ to $(D,o)$}: 
\begin{align*}
&T_{(D,[C]_{\rm ori})\to (D,o)}: \{D\} \times {\rm Col}_{(D_X,\rpm)_Y} (D)/\sim_{\rm ori} \to \{D\} \times {\rm ori}(D);\\
&(D, [C]_{\rm ori}) \mapsto \mbox{the oriented diagram $(D,o)$ induced by the canonical form of $(D,C)$},
\end{align*}
see Figure~\ref{CO-map2}.
\begin{figure}[t]
 \begin{center}
\includegraphics[width=105mm]{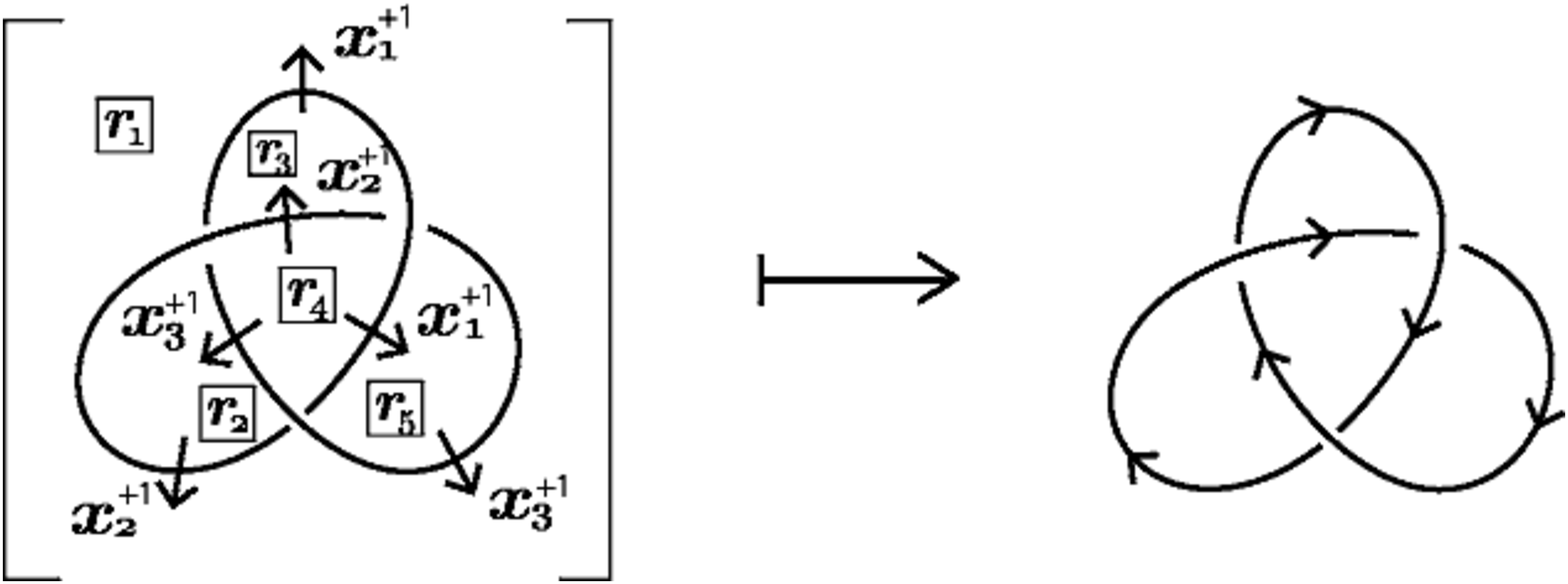}
\end{center}
\caption{The map $T_{(D,[C]_{\rm ori}) \to (D,o)}$}
\label{CO-map2}
\end{figure}
We then have the following property.
\begin{lemma}
$T_{(D,[C]_{\rm ori})\to (D,o)}$ is bijective.
\end{lemma}
Any $(D_X,\rpm)_Y$-colored diagram $(D, C)$ uniquely induces an oriented, $X_Y$-colored diagram $((D,o) , \bar{C})$ by setting on $D$ the orientation induced by the canonical form of $(D,C)$, replacing, for the canonical form of $C$, the assigned element $x^{+1} \in X^{+1}$ of each (semi-)arc with $x \in X$, and leaving the assignment by $Y$ as it is, see Figure~\ref{CD-map2}.
\begin{figure}[t]
 \begin{center}
\includegraphics[width=110mm]{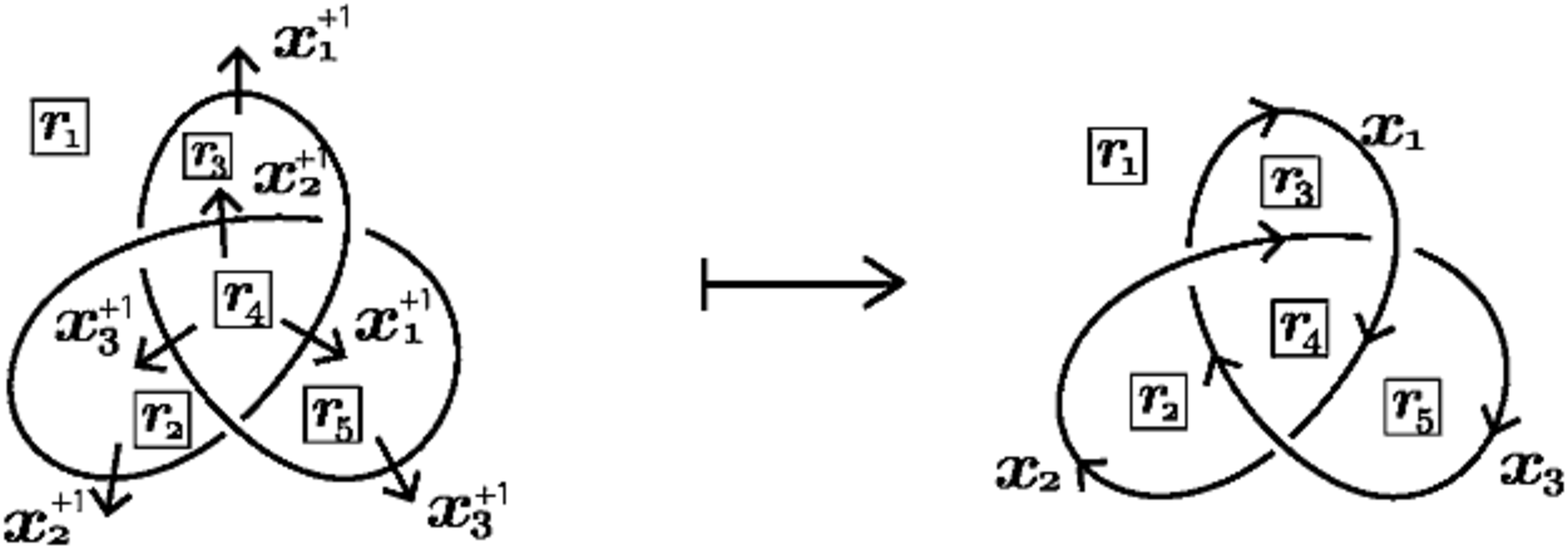}
\end{center}
\caption{The map $T_{(D,C) \to ((D,o), C)}$}
\label{CD-map2}
\end{figure}
We call such $((D,o), \bar{C})$ the {\it oriented, $X_Y$-colored diagram induced by the canonical form of $(D,C)$}.
The {\it transformation map from $(D,C)$ to $((D,o), C)$} is the map
\begin{align*}
&T_{(D,C) \to ((D,o), C)}:   \{D\} \times {\rm Col}_{(D_X,\rpm)_Y}(D) \to  \bigcup_{o\in {\rm ori}(D)} \{(D,o)\}\times  {\rm Col}_{X_Y}(D,o);\\
&(D,C) \mapsto \mbox{ the oriented, $X_Y$-colored diagram induced by the canonical form of $(D,C)$,}
\end{align*}
 see Figure~\ref{CD-map2}.
\begin{lemma}
$T_{(D,C) \to ((D,o), C)}$ is bijective.
\end{lemma}
 We have the following theorem.
\begin{theorem}\label{thm:correspondencecoloring}
Let $(D,o)$ be a diagram of an oriented link and take $[C]_{\rm ori} \in {\rm Col}_{(D_X,\rpm)_Y} (D)/\sim_{\rm ori} $ such that 
$(D,o)  =T_{(D,[C]_{\rm ori})\to (D,o)} (D, [C]_{\rm ori} )$.
Then there exists a one-to-one correspondence between ${\rm Col}_{X_Y}(D,o)$ and $[C]_{\rm ori}$, and thus, we have
\[
\# {\rm Col}_{X_Y}(D,o) = \# [C]_{\rm ori}.
\]
\end{theorem}
\begin{proof}
As in the proof of Theorem~\ref{lem:correspondenceCorandori2},
the bijection  is given by the restriction of the translation map $T_{(D,C) \to ((D,o), C)}$
\begin{align*}
&T_{(D,C) \to ((D,o), C)} |_{\{D\} \times [C]_{\rm ori}} : \{D\} \times [C]_{\rm ori} \to \{(D,o)\}\times  {\rm Col}_{ X_Y }(D,o). 
\end{align*}
\end{proof}


\subsection{An interpretation of quandle cocycle invariants by using symmetric quandle cocycle invariants}
Let $X$ be a quandle, $(D_X, \rpm)$ the symmetric double of $X$, and $Y$ an $X$-set.
We note again that by Lemma~\ref{lem:Y}, we can also regard $Y$ as a $(D_X,\rpm)$-set.

For a diagram $D$ of an unoriented link and $[C]_{\rm ori} \in {\rm Col}_{(D_X,\rpm)_Y}(D)/\sim_{\rm ori}$, we denote by $\mathcal{W}^{(D_X,\rpm)_Y} (D, [C]_{\rm ori})$ the multi-set
\[
\{
[W^{(D_X,\rpm)_Y} (D,C)] \in H_2^{\rm SQ} (D_X,\rpm)_Y ~|~ C \in [C]_{\rm ori}
\},
\]
where for  a $(D_X,\rpm)_Y$-colored diagram  $(D,C)$, we denote by $W^{(D_X,\rpm)_Y}(D,C)$ the sum of the weights of all the crossings as in Subsection~\ref{Symmetric quandle cocycle invariants of (unoriented) links}, and  note that $\mathcal{W}^{(D_X,\rpm)_Y} (D, [C]_{\rm ori}) \subset \mathcal{W}^{(D_X,\rpm)_Y} (D)$.
The next theorem shows that the quandle homology invariant $\mathcal{W}^{X_Y} (L,o)$ of an oriented link $(L,o)$ can be also interpreted by using symmetric quandles.

\begin{theorem}\label{thm:weightsum2}
Let $(D,o)$ is a diagram of an oriented link and take $[C]_{\rm ori} \in {\rm Col}_{(D_X,\rpm)_Y}(D)/\sim _{\rm ori}$ 
such that $(D,o)=T_{(D,[C]_{\rm ori}) \to (D,o) }(D, [C]_{\rm ori})$.
Then as multi-sets, we have 
\[
(T_2^{D_X \to X})^*\big(\mathcal{W}^{(D_X,\rpm)_Y} (D, [C]_{\rm ori})\big)=\mathcal{W}^{X_Y} (D,o).
\]
\end{theorem}
\begin{proof}
Note that as mentioned in the proof of Theorem~\ref{thm:correspondencecoloring}, 
the restriction 
\begin{align*}
&T_{(D,C) \to ((D,o), C)} |_{\{D\} \times [C]_{\rm ori}} : \{D\} \times [C]_{\rm ori} \to \{(D,o)\}\times  {\rm Col}_{ X_Y }(D,o) 
\end{align*}
is a bijection.

For $C\in [C]_{\rm ori}$, let $\bar{C}\in {\rm Col}_{ X_Y }(D,o)$ such that $T_{(D,C) \to ((D,o), C)} (D,C)= ((D,o), \bar{C})$.
Since $\bar{C}$ is obtained from $C$ by giving on $D$ the orientation induced by the canonical form of $(D,C)$, replacing, for the canonical form of $C$, the assigned element $x^{+1} \in X^{+1}$ of each (semi-)arc with $x \in X$, and leaving the assignment by $Y$ as it is,  when we have a weight $w_\chi=\varepsilon (r,x_1^{+1},x_2^{+1})$ for a crossing $\chi$ of $(D,C)$, we have the weight $\bar{w}_\chi=\varepsilon (r,x_1,x_2)$  for the same crossing $\chi$ of $((D,o), \bar{C})$, see Figure~\ref{weight1}.
This implies that 
$T_2^{D_X \to X} (W^{(D_X,\rpm)_Y} (D, C)) = W^{X_Y} ((D,o),\bar{C})$, and moreover,
$(T_2^{D_X \to X})^* ([W^{(D_X,\rpm)_Y} (D, C)]) = [W^{X_Y} ((D,o),\bar{C})]$, and thus, we have $(T_2^{D_X \to X})^* (\mathcal{W}^{(D_X,\rpm)_Y} (D, [C]_{\rm ori}))=\mathcal{W}^{X_Y} (D,o)$.
\begin{figure}[t]
 \begin{center}
\includegraphics[width=90mm]{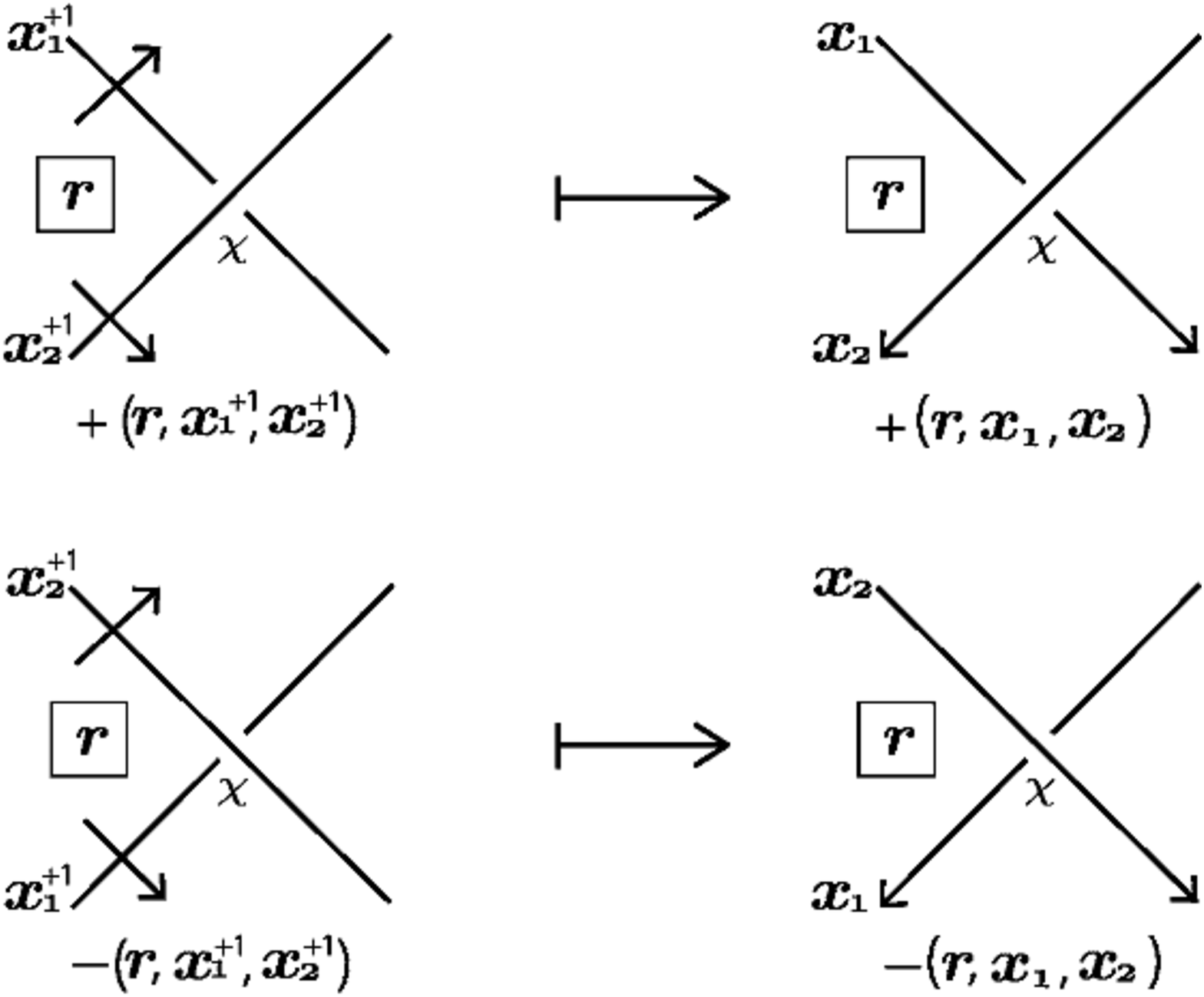}
\end{center}
\caption{The weight of a crossing}
\label{weight1}
\end{figure}
\end{proof}

For an unoriented link $L$, by considering all the orientations of $L$ and taking the quandle homology invariant  $\mathcal{W}^{X_Y} (L,o)$ for each orientation $o$, we can define an invariant of $L$. 
More precisely, for a diagram $D$ of $L$, the multi-set 
\[
\{ \mathcal{W}^{X_Y} (D,o) ~|~ o \in {\rm ori}(D) \} 
\]
is an invariant of $L$.
The next property implies that the quandle homology  invariants of an unoriented link $L$ can be  also interpreted by using symmetric quandles.
\begin{corollary}\label{corweightsum}
Let $D$ be a diagram of an unoriented link.
As multi-sets, we have
\[
\begin{array}{l}
\{ \mathcal{W}^{X_Y} (D,o) ~|~ o \in {\rm ori}(D) \}  \\[3pt]
=\{(T_2^{D_X \to X})^* (\mathcal{W}^{(D_X,\rpm)_Y} (D, [C]_{\rm ori}))~|~ [C]_{\rm ori} \in {\rm Col}_{(D_X,\rpm)_Y} (D)/\sim_{\rm ori} \}.
\end{array}
\]
\end{corollary}

Let $A$ be an abelian group.
Let $\theta^{\rm Q}: C^{\rm Q}_2 (X)_Y \to A$ and $\theta^{\rm SQ}: C^{\rm SQ}_2 (D_X,\rpm)_Y \to A$ be a quandle $2$-cocycle and a symmetric quandle $2$-cocycle, respectively,  such that $\theta^{\rm Q} \circ T_2^{D_X \to X} = \theta^{\rm SQ} $.

For a diagram $D$ of an unoriented link and $[C]_{\rm ori} \in {\rm Col}_{(D_X,\rpm)_Y}(D)/\sim_{\rm ori}$, we denote by ${\Phi}_{\theta^{\rm SQ}}^{(D_X,\rpm)_Y}(D,[C]_{\rm ori})$ the multi-set
\[
\{
\theta^{\rm SQ} (W^{(D_X,\rpm)_Y} (D,C))\in A  ~|~ C \in [C]_{\rm ori}
\}.
\]
We note that ${\Phi}_{\theta^{\rm SQ}}^{(D_X,\rpm)_Y}(D,[C]_{\rm ori}) \subset {\Phi}_{\theta^{\rm SQ}}^{(D_X,\rpm)_Y}(D)$.

The next theorem shows that the quandle cocycle invariant $\Phi_{\theta^{\rm Q}}^{X_Y}(L,o)$ of an  oriented link $(L,o)$ can be also interpreted by using  symmetric quandles.
\begin{theorem}\label{thm:quandlecocycleinvariantsinterpretation2}
Let $(D,o)$ be a diagram of an oriented link and take $[C]_{\rm ori} \in {\rm Col}_{(D_X,\rpm)_Y}(D)/\sim _{\rm ori}$ 
such that $(D,o)=T_{(D,[C]_{\rm ori}) \to (D,o) }(D, [C]_{\rm ori})$.
As multi-sets, we have 
\[
\Phi_{\theta^{\rm Q}}^{X_Y}(D,o)={\Phi}_{\theta^{\rm SQ}}^{(D_X,\rpm)_Y}(D,[C]_{\rm ori}).
\]
\end{theorem}
\begin{proof}
Note again that as mentioned in the proof of Theorem~\ref{thm:correspondencecoloring}, 
the restriction 
\begin{align*}
&T_{(D,C) \to ((D,o), C)} |_{\{D\} \times [C]_{\rm ori}} : \{D\} \times [C]_{\rm ori} \to \{(D,o)\}\times  {\rm Col}_{ X_Y }(D,o) 
\end{align*}
is a bijection.

For $C\in [C]_{\rm ori}$, let $\bar{C}\in {\rm Col}_{ X_Y }(D,o)$ such that $T_{(D,C) \to ((D,o), C)} (D,C)= ((D,o), \bar{C})$.
Since we have $T_2^{D_X \to X} (W^{(D_X,\rpm)_Y} (D, C)) = W^{X_Y} ((D,o),\bar{C})$ as mentioned in the proof of  Theorem~\ref{thm:weightsum2},  it holds that
\begin{align*}
&{\Phi}_{\theta^{\rm SQ}}^{(D_X,\rpm)_Y}(D,[C]_{\rm ori})&&=\{ \theta^{\rm SQ} (W^{(D_X,\rpm)_Y} (D,C))\in A  ~|~ C \in [C]_{\rm ori} \}\\
&&&=\{ \theta^{\rm Q} \circ T_2^{D_X \to X} (W^{(D_X,\rpm)_Y} (D,C))\in A  ~|~ C \in [C]_{\rm ori} \}\\
&&&=\{ \theta^{\rm Q} (W^{X_Y} ((D,o),\bar{C}))\in A  ~|~ \bar{C} \in {\rm Col}_{X_Y}(D,o) \}\\
&&&=\Phi_{\theta^{\rm Q}}^{X_Y} (D,o).
\end{align*}
\end{proof}

For an unoriented link $L$, by considering all the orientations of $L$ and taking the quandle cocycle invariant $\Phi_\theta ^{X_Y} (D,o)$  for each orientation $o$, we can define an invariant of $L$. 
More precisely, for an diagram $D$ of $L$, the multi-set 
\[
\{ \Phi_\theta^{X_Y} (D,o) ~|~ o \in {\rm ori}(D) \} 
\]
is an invariant of $L$.  
The next property implies that the quandle cocycle invariants of an unoriented link $L$ can be  also interpreted by using symmetric quandles.
\begin{corollary}\label{corcocycleinvariants}
Let $D$ be a diagram of an unoriented link.
As multi-sets, we have
\begin{align*}
&\{ \Phi_{\theta^{\rm Q}}^{X_Y} (D,o) ~|~ o \in {\rm ori}(D) \} \\
&= \{{\Phi}_{\theta^{\rm SQ}}^{(D_X,\rpm)_Y}(D,[C]_{\rm ori})~|~ [C]_{\rm ori} \in {\rm Col}_{(D_X,\rpm)_Y} (D)/\sim  _{\rm ori} \}.
\end{align*}
\end{corollary} 

\begin{remark}
As a consequence of the results of this section, it is no exaggeration to say that  symmetric quandles are also useful for oriented links, and  
symmetric quandles  are not less than quandles in the case of considering coloring numbers and cocycle invariants for oriented links.
\end{remark}

\section{Invariants for surface-links using quandles and symmetric quandles}\label{sec:6}
For oriented (or orientable, unoriented) surface-links, we have the same properties, that is, quandle coloring numbers, quandle homology invariants and quandle cocycle invariants are interpreted by using symmetric doubles of quandles. 
In this section, we summarize our results for surface-links.

\subsection{Surface-links}
A {\it (orientable) surface-link} is a disjoint union of orientable closed surfaces locally-flatly embedded in $\mathbb R^4$.
Two surface-links are said to be {\it equivalent} if there exists an orientation-preserving self-homeomorphism of $\mathbb R^4$ which maps one surface-link onto the other.
A {\it diagram} of a surface-link is its image, via a generic projection from $\mathbb R^4$ to $\mathbb R^3$, equipped with the height information around the double point curves. 
The height information is given by removing the regular neighborhoods of the lower double points. 
Then a diagram is regarded as a disjoint union of connected compact surfaces, each of which is called a {\it sheet}.
A connected component after removing the regular neighborhoods of all double points from the diagram is called a {\it semi-sheet}.
The $3$-dimensional space $\mathbb R^3$ is separated into several connected regions by a diagram. We call each connected region a {\it complementary region} of the diagram.
Refer to \cite{CarterJelsovskyKamadaLangfordSaito03, CarterKamadaSaitobook, CarterSaitobook, KamadaOshiro10} for more details. 
For an oriented surface-link diagram $(D,o)$, we often use a normal orientation $n_{s}$ for a sheet $s$  to represent the orientation $o=(o_1,o_2)$ such that the triple $(o_1,o_2,n_s)$ matches the right-handed orientation of $\mathbb R^3$, see Figure~\ref{orientation}.  
\begin{figure}[t]
 \begin{center}
\includegraphics[width=90mm]{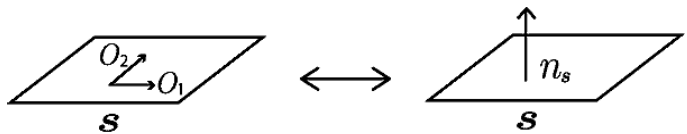}
\end{center}
\caption{The orientation and the normal orientation}
\label{orientation}
\end{figure}

\subsection{Coloring numbers using quandles and symmetric quandles for oriented surface-links}
An {\it $X$-coloring} of an oriented surface-link diagram $(D,o)$ for a given quandle $(X,*)$ is defined as an assignment  of an element of $X$ to each sheet satisfying the  condition depicted in the left of Figure~\ref{doublepointcondition} around double point curves.
We denote by ${\rm Col}_{X}(D,o)$ the set of $X$-colorings of $(D,o)$. 
An {\it $X_Y$-coloring} of an oriented surface-link diagram $(D,o)$ for a given quandle $(X,*)$ and a given $X$-set $Y$ is an $X$-coloring of $(D,o)$ with an assignment of an element of $Y$ to each complementary region of $D$ satisfying the condition depicted in the right of Figure~\ref{doublepointcondition} around semi-sheets.
We denote by ${\rm Col}_{X_Y}(D,o)$ the set of $X_Y$-colorings of $(D,o)$. 
Refer to \cite{CarterJelsovskyKamadaLangfordSaito03, CarterKamadaSaitobook} for more details. 
\begin{figure}[t]
 \begin{center}
\includegraphics[width=90mm]{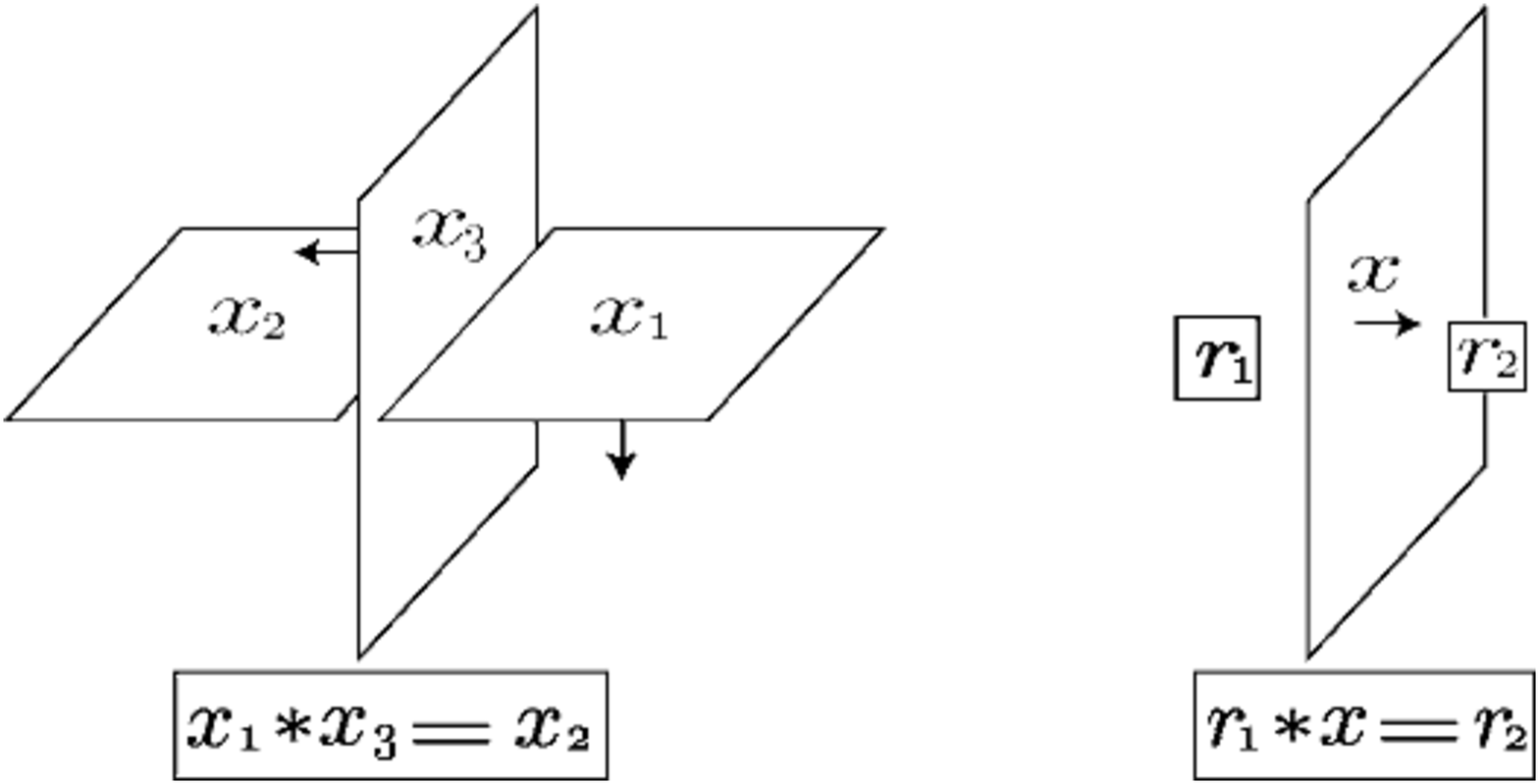}
\end{center}
\caption{The conditions of colorings}
\label{doublepointcondition}
\end{figure}

An {\it $(X,\rho)$-coloring} of a surface-link diagram $D$ for a given symmetric quandle $(X,\rho,*)$ is defined as an equivalence class of an assignment  of an element of $X$  and a normal orientation to each semi-sheet satisfying the conditions depicted in Figures~\ref{doublepointcondition} and \ref{doublepointcondition2} around double point curves, where the equivalence relation is generated by basic inversions depicted in Figure~\ref{basicinversion2}.
We denote by ${\rm Col}_{(X,\rho)}(D)$ the set of $(X,\rho)$-colorings of $D$. 
An {\it $(X,\rho)_Y$-coloring} of a surface-link diagram $D$ for a given symmetric quandle $(X,\rho,*)$ and a given $(X,\rho)$-set $Y$ is an $(X,\rho)$-coloring of $D$ with an assignment of an element of $Y$ to each complementary region of $D$ satisfying the condition depicted in the right of Figure~\ref{doublepointcondition} around semi-sheets. We denote by ${\rm Col}_{(X,\rho)_Y}(D)$ the set of $(X,\rho)_Y$-colorings of $D$.
Refer to \cite{KamadaOshiro10} for more details.
\begin{figure}[t]
 \begin{center}
\includegraphics[width=95mm]{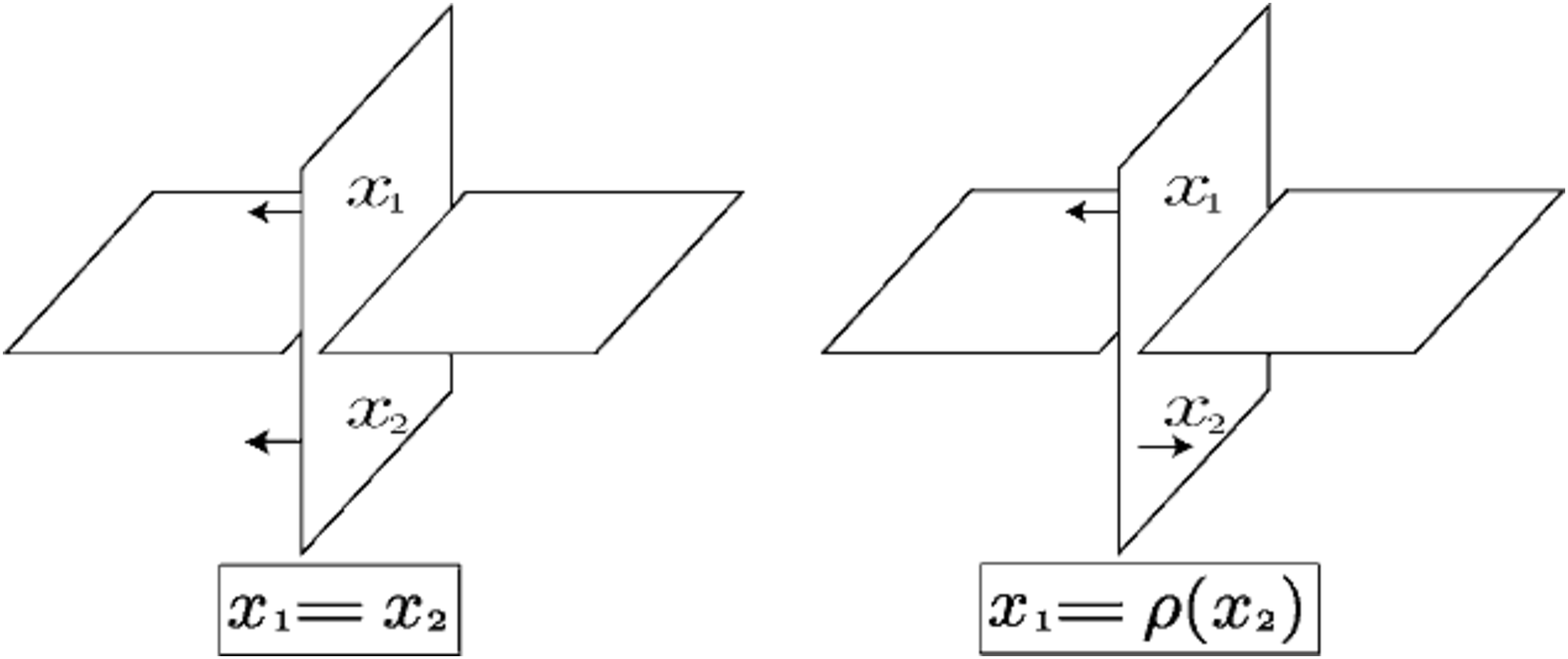}
\end{center}
\caption{The conditions of colorings}
\label{doublepointcondition2}
\end{figure}
\begin{figure}[t]
 \begin{center}
\includegraphics[width=90mm]{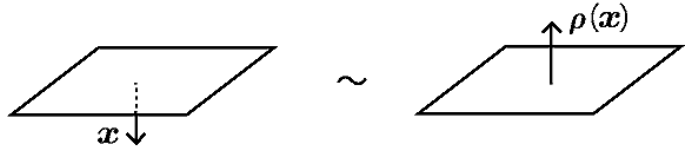}
\end{center}
\caption{A basic inversion}
\label{basicinversion2}
\end{figure}

Let $X$ be a quandle and  $(D_X, \rpm)$ the symmetric double of $X$.
As in the case of links, by performing basic inversions, any $(D_X,\rpm)$-coloring  $C$ of a diagram $D$ of an unoriented surface-link is uniquely represented by an assignment of a normal orientation  and an element of $X^{+1}$ to each semi-sheet of $D$.  We call the assignment the {\it canonical form} of $C$.

Let $D$ be a diagram of an unoriented surface-link and  $C$ a $(D_X,\rpm)$-coloring  of $D$.
For a semi-sheet $s$ of $D$ and the normal orientation $n_s$  assigned to $s$ for the canonical form of $C$,  
we set the orientation $(o_1, o_2)$ of $s$ so that the triple $(o_1, o_2, n_s)$ matches the right-handed orientation of $\mathbb R^3$, see Figure~\ref{orientation}. 
Such orientations $(o_1, o_2)$ determine an orientation $o=(o_1,o_2)$ of $D$, and we call the orientation the {\it orientation induced by the canonical form of $(D,C)$} and the oriented diagram $(D,o)$ the {\it oriented diagram induced by the canonical form of $(D,C)$}.
Let $C$ and $C'$ be $(D_X,\rpm)$-colorings of $D$.
We say that $C$ and $C'$ are {\it orientation equivalent} $(C\sim_{\rm ori} C')$ if the orientation induced by the canonical form of $(D,C)$ coincides with that of $(D,C')$. We denote by $[C]_{\rm ori}$ the orientation equivalence class of a $(D_X,\rpm)$-coloring $C$.
The {\it transformation map from $(D,[C]_{\rm ori})$ to $(D,o)$} is the map  
\begin{align*}
&T_{(D,[C]_{\rm ori})\to (D,o)}: \{D\} \times {\rm Col}_{(D_X,\rpm)} (D)/\sim_{\rm ori} \to \{D\} \times {\rm ori}(D);\\
&(D, [C]_{\rm ori}) \mapsto \mbox{the oriented diagram $(D,o)$ induced by the canonical form of $(D,C)$}.
\end{align*}

The next theorem implies that quandle coloring numbers for oriented surface-links are interpreted by using symmetric quandle colorings.
\begin{theorem}\label{lem:correspondenceCorandori2surface}
Let $(D,o)$ be a diagram of an oriented surface-link and take $[C]_{\rm ori} \in {\rm Col}_{(D_X,\rpm)} (D)/\sim_{\rm ori} $ such that $(D,o)=T_{(D,[C]_{\rm ori})\to (D,o)}(D,[C]_{\rm ori})$.
Then there exists a one-to-one correspondence between ${\rm Col}_X(D,o)$ and $[C]_{\rm ori}$, and thus, we have
\[
\# {\rm Col}_X(D,o) = \# [C]_{\rm ori}.
\]
\end{theorem}
 The next corollary implies that the quandle coloring numbers of an (orientable, unoriented) surface-link $F$ can be also interpreted by using symmetric quandles. 
\begin{corollary}\label{corsurface}
Let $D$ be a diagram of an unoriented surface-link.
As multi-sets, we have
\[
\{ \# {\rm Col}_{X} (D,o) ~|~ o  \in {\rm ori}(D) \} = \{  \# [C]_{\rm ori} ~|~ [C]_{\rm ori} \in {\rm Col}_{(D_X,\rpm)}(D)/\sim_{\rm ori}\}.
\]
\end{corollary}

Similarly we have the following {\it transformation map from $(D,[C]_{\rm ori})$ to $(D,o)$}.
\begin{align*}
&T_{(D,[C]_{\rm ori})\to (D,o)}: \{D\} \times {\rm Col}_{(D_X,\rpm)_Y} (D)/\sim_{\rm ori} \to \{D\} \times {\rm ori}(D);\\
&(D, [C]_{\rm ori}) \mapsto \mbox{the oriented diagram $(D,o)$ induced by the canonical form of $(D,C)$},
\end{align*}
where we use the same symbol $T_{(D,[C]_{\rm ori})\to (D,o)}$ for simplicity.
We have the following theorem.
\begin{theorem}\label{thmsurface}
Let $(D,o)$ be a diagram of an oriented surface-link and take $[C]_{\rm ori} \in {\rm Col}_{(D_X,\rpm)_Y} (D)/\sim_{\rm ori} $ such that $(D,o)=T_{(D,[C]_{\rm ori})\to (D,o)}(D,[C]_{\rm ori})$.
Then there exists a one-to-one correspondence between ${\rm Col}_{X_Y}(D,o)$ and $[C]_{\rm ori}$, and thus, we have
\[
\# {\rm Col}_{X_Y}(D,o) = \# [C]_{\rm ori}.
\]
\end{theorem}

\begin{remark}
Theorem~\ref{lem:correspondenceCorandori2surface}, Corollary~\ref{corsurface} and Theorem~\ref{thmsurface} are analogous to Theorem~\ref{lem:correspondenceCorandori2}, Corollary~\ref{thm:coloringnumber} and Theorem~\ref{thm:correspondencecoloring}, respectively.
\end{remark}

\subsection{Cocycle invariants using quandles and symmetric quandles  for oriented surface-links}
The {\it quandle homology invariant} $\mathcal{W}^{X_Y} (F,o)$ of an oriented surface-link $(F,o)$ for a quandle $X$ and an $X$-set $Y$ is defined by taking the sum of the weights of triple points for each $X_Y$-coloring, and collecting the weight sums for all the $X_Y$-colorings, refer to \cite{CarterJelsovskyKamadaLangfordSaito03, CarterKamadaSaitobook} for more details, where the weight of a triple point is obtained as depicted in Figure~\ref{triplepointweight}.
The {\it symmetric quandle homology invariant} $\mathcal{W}^{(X,\rho)_Y} (F)$ of an (unoriented) surface-link $F$ for a symmetric quandle $(X,\rho)$ and  an $(X,\rho)$-set $Y$  is also defined by taking the sum of the weights of triple points for each $(X,\rho)_Y$-coloring, and collecting the weight sums for all the $(X,\rho)_Y$-colorings, refer to \cite{KamadaOshiro10}  for more details, where the weight of a triple point is obtained as depicted in Figure~\ref{triplepointweight}.
\begin{figure}[t]
 \begin{center}
\includegraphics[width=90mm]{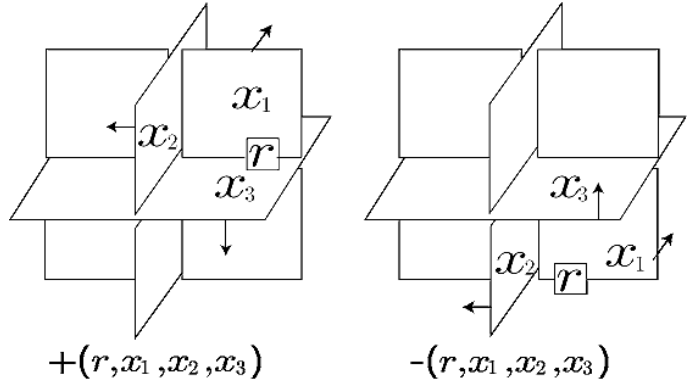}
\end{center}
\caption{The weight of a triple point}
\label{triplepointweight}
\end{figure}

The {\it quandle cocycle invariant} $\Phi^{X_Y}_\theta (F,o)$ of an oriented surface-link $(F,o)$ for a quandle $X$, an $X$-set $Y$ and a quandle $3$-cocycle $\theta$ is defined by taking the sum of the weights of triple points for each $X_Y$-coloring, and setting the multi-set of the images, by $\theta$, of the weight sums for all the $X_Y$-colorings, refer to \cite{CarterJelsovskyKamadaLangfordSaito03, CarterKamadaSaitobook}.
The {\it symmetric quandle cocycle invariant} $\Phi^{(X,\rho)_Y}_\theta (F)$ of an (unoriented) surface-link $F$ for a symmetric quandle $(X,\rho)$, an $(X,\rho)$-set $Y$ and a symmetric quandle $3$-cocycle $\theta$ is also defined by taking the sum of the weights of triple points for each $(X,\rho)_Y$-coloring, and setting the multi-set of the images, by $\theta$, of the weight sums for all the $(X,\rho)_Y$-colorings, refer to \cite{KamadaOshiro10}.

Let $X$ be a quandle, $(D_X, \rpm)$ the symmetric double of $X$, 
and $Y$  an $X$-set.  By Lemma~\ref{lem:Y}, we can also regard $Y$ as a $(D_X,\rpm)$-set.
We have the next theorem, which  shows that the quandle homology invariant $\mathcal{W}^{X_Y} (F,o)$ of an oriented surface-link $(F,o)$ can be interpreted by using symmetric quandles. 
\begin{theorem}\label{surface1}
Let $(D,o)$ is a diagram of an oriented surface-link and take $[C]_{\rm ori} \in {\rm Col}_{(D_X,\rpm)_Y}(D)/\sim _{\rm ori}$  such that $(D,o)=T_{(D,[C]_{\rm ori}) \to (D,o) }(D, [C]_{\rm ori})$.
Then as multi-sets, we have 
\[
(T_3^{D_X \to X})^*\big(\mathcal{W}^{(D_X,\rpm)_Y} (D, [C]_{\rm ori})\big)=\mathcal{W}^{X_Y} (D,o).
\]
\end{theorem}
 The next corollary shows that the quandle homology invariant $\mathcal{W}^{X_Y} (F)$ of an (orientable, unoriented) surface-link $F$ can be also interpreted by using symmetric quandles. 
\begin{corollary}\label{surface2}
Let $D$ be a diagram of an unoriented surface-link.
As multi-sets, we have
\[
\begin{array}{l}
\{ \mathcal{W}^{X_Y} (D,o) ~|~ o \in {\rm ori}(D) \}  \\[3pt]
=\{(T_3^{D_X \to X})^* (\mathcal{W}^{(D_X,\rpm)_Y} (D, [C]_{\rm ori}))~|~ [C]_{\rm ori} \in {\rm Col}_{(D_X,\rpm)_Y} (D)/\sim_{\rm ori} \}.
\end{array}
\]
\end{corollary}

Let $A$ be an abelian group.
Let $\theta^{\rm Q}: C^{\rm Q}_3 (X)_Y \to A$ and $\theta^{\rm SQ}: C^{\rm SQ}_3 (D_X,\rpm)_Y \to A$ be a quandle $3$-cocycle and a symmetric quandle $3$-cocycle, respectively,  such that $\theta^{\rm Q} \circ T_3^{D_X \to X} = \theta^{\rm SQ} $.
The next theorem shows that the quandle cocycle invariant $\Phi_{\theta^{\rm Q}}^{X_Y}(F,o)$ of an  oriented surface-link $(F,o)$ can be also interpreted by using  symmetric quandles.
\begin{theorem}\label{surface3}
Let $(D,o)$ is a diagram of an oriented surface-link and take $[C]_{\rm ori} \in {\rm Col}_{(D_X,\rpm)_Y}(D)/\sim _{\rm ori}$ 
such that $(D,o)=T_{(D,[C]_{\rm ori}) \to (D,o) }(D, [C]_{\rm ori})$.
As multi-sets, we have 
\[
\Phi_{\theta^{\rm Q}}^{X_Y}(D,o)={\Phi}_{\theta^{\rm SQ}}^{(D_X,\rpm)_Y}(D,[C]_{\rm ori}).
\]
\end{theorem}
 The next corollary shows that the quandle cocycle invariant  $\Phi_{\theta^{\rm Q}}^{X_Y}(F)$  of an (orientable, unoriented) surface-link $F$ can be also interpreted by using symmetric quandles. 
\begin{corollary}\label{surface4}
Let $D$ be a diagram of an unoriented surface-link.
As multi-sets, we have
\begin{align*}
&\{ \Phi_{\theta^{\rm Q}}^{X_Y} (D,o) ~|~ o \in {\rm ori}(D) \} \\
&= \{{\Phi}_{\theta^{\rm SQ}}^{(D_X,\rpm)_Y}(D,[C]_{\rm ori})~|~ [C]_{\rm ori} \in {\rm Col}_{(D_X,\rpm)_Y} (D)/\sim  _{\rm ori} \}.
\end{align*}
\end{corollary} 

\begin{remark}
Theorem~\ref{surface1}, Corollary~\ref{surface2}, Theorem~\ref{surface3} and  Corollary~\ref{surface4} are analogous to Theorem~\ref{thm:weightsum2}, Corollary~\ref{corweightsum}, Theorem~\ref{thm:quandlecocycleinvariantsinterpretation2} and  Corollary~\ref{corcocycleinvariants}, respectively.
\end{remark}

\begin{remark}
As a consequence of the results of this section, it is no exaggeration to say that  symmetric quandles are also useful for oriented surface-links, and  
symmetric quandles are not less than quandles in the case of considering coloring numbers and cocycle invariants for oriented surface-links.
\end{remark}


\section*{Acknowledgments}

This work was supported by JSPS KAKENHI Grant Number 16K17600.

\end{document}